\documentclass[leqno,12pt]{amsart} 

\usepackage{amssymb} 
\usepackage{amsmath}
\usepackage{amsthm}
\usepackage{amscd}
\usepackage{bm}
\usepackage{graphicx,psfrag,wrapfig}
\theoremstyle{plain} 
\newtheorem{theorem}{Theorem}[section]  
\newtheorem{lemma}[theorem]{Lemma}
\newtheorem{corollary}[theorem]{Corollary}
\newtheorem{proposition}[theorem]{Proposition}

\newtheorem{sublemma}[theorem]{Sublemma}
\theoremstyle{definition} 
\newtheorem{definition}[theorem]{Definition}
\newtheorem{remark}[theorem]{Remark}
\newtheorem{example}[theorem]{Example}
\newtheorem{notation}[theorem]{Notation}

\newcommand{\const}{\mathrm{const}}
\newcommand{\norm}[1]{\left|\!\left|#1\right|\!\right|}
\newcommand{\ad}{\mathrm{ad}}
\newcommand{\supp}{\mathrm{supp}}
\newcommand{\vol}{\mathrm{vol}}
\newcommand{\Diam}{\mathrm{Diam}}
\newcommand{\moduli}{\mathcal{M}}
\newcommand{\widim}{\mathrm{Widim}}
\newcommand{\dist}{\mathrm{dist}}

\begin{document}

\title[Local mean dimension of ASD moduli spaces over the cylinder]
{Local mean dimension of ASD moduli spaces over the cylinder} 

\author[S. Matsuo and M. Tsukamoto]{Shinichiroh Matsuo and Masaki Tsukamoto}



\date{\today}

\maketitle

\begin{abstract}
We study an infinite dimensional ASD moduli space over the cylinder.
Our main result is the formula of its local mean dimension.
A key ingredient of the argument is the notion of non-degenerate ASD connections.
We develop its deformation theory and show that there exist sufficiently many non-degenerate ASD connections
by using the method of gluing infinitely many instantons.
\end{abstract}


\section{Introduction} \label{section: Introduction}
\subsection{Main result} \label{subsection: main result}

This paper is a continuation of \cite{Matsuo-Tsukamoto}.
(But readers don't need a knowledge of \cite{Matsuo-Tsukamoto}.)
We study a certain infinite dimensional ASD moduli space over the cylinder $\mathbb{R}\times S^3$.
The main motivation is to develop an infinite dimensional analogue of the pioneering work of 
Atiyah--Hitchin--Singer \cite{A-H-S}.
The paper \cite{A-H-S} is a starting point of the mathematical study of Yang--Mills gauge theory.
One of their main results \cite[Theorem 6.1]{A-H-S} is a calculation of the dimension of an ASD moduli space 
by using the Atiyah--Singer index theorem.
Their result can be stated as follows:
Let $A$ be an irreducible $SU(2)$ ASD connection over a compact anti-self-dual 4-manifold of positive scalar curvature.
Then the number of the 
parameters of its deformation is 
\[ 8(\text{instanton number of $A$}) - 3(1-b_1).\]
Here $b_1$ is the first Betti number of the underlying 4-manifold.
The ``instanton number'' means the second Chern number of the bundle which the connection $A$ belongs to, and 
it is equal to the Yang--Mills functional 
\[ \frac{1}{8\pi^2}\int |F_A|^2d\vol.\]
This dimension formula is the target of our work.
Our main result (Theorem \ref{thm: main theorem}) is an infinite dimensional analogue of the above formula.
Although there is still much work to be done, probably our theorem is the first satisfactory result in this direction.

Let $S^3:= \{x_1^2+x_2^2+x_3^2+x_4^2=1\}\subset \mathbb{R}^4$ be the unit 3-sphere with the 
Riemannian metric induced by the Euclidean metric on $\mathbb{R}^4$.
Let $X:=\mathbb{R}\times S^3$ be the cylinder with the product metric.
The reason why we consider $\mathbb{R}\times S^3$ is as follows:
In \cite[Theorem 6.1]{A-H-S} they needed the assumption that the underling 4-manifold is anti-self-dual and 
has positive scalar curvature.
These metrical conditions were used via a certain Weitzenb\"{o}ck formula.
In the present paper we also need to use the Weitzenb\"{o}ck formula several times.
The cylinder $\mathbb{R}\times S^3$ is one of the simplest non-compact 4-manifolds which are
anti-self-dual and has uniformly positive scalar curvature.
We need these metrical conditions.

Let $E:=X\times SU(2)$ be the product principal $SU(2)$ bundle.
(Every principal $SU(2)$ bundle on $X$ is isomorphic to the product bundle $E$.)
Let $A$ be a connection on $E$.
Its curvature
$F_A$ is a 2-form valued in the adjoint bundle $\ad E = X\times su(2)$.
So it gives a linear map:
\[ F_{A}(p):\Lambda^2(T_pX) \to su(2) \quad (\forall p\in X).\] 
Let $|F_{A}(p)|_{\mathrm{op}}$ be the \textit{operator norm} of this linear map, and 
let $\norm{F_A}_{\mathrm{op}}$ be the supremum of $|F_{A}(p)|_{\mathrm{op}}$ over $p\in X$.
The explicit formula is as follows:
Let $p\in X$, and let $(x_1,x_2,x_3,x_4)$ be the normal coordinate system of $X$ 
centered at $p$.
We suppose that the curvature $F_A$ is expressed by $F_A =\sum_{1\leq i<j\leq 4} F_{ij}dx_i\wedge dx_j$ around the point $p$.
Then the norm $|F_{A}(p)|_{\mathrm{op}}$ is equal to 
\[  \sup\left\{\left|\sum_{1\leq i<j\leq 4}a_{ij}F_{ij}(p)\right||\, a_{ij}\in \mathbb{R}, 
   \sum_{1\leq i<j\leq 4} a_{ij}^2=1 \right\}. \]
Here the Lie algebra $su(2) = \{X\in M_2(\mathbb{C})|\, X+X^*=0, \,\mathrm{tr} (X)=0\}$
is endowed with the inner product $\langle X,Y\rangle = -\mathrm{tr}(XY)$.
In this paper we also use the Euclidean norm $|F_{A}(p)|$ defined by 
\begin{equation} \label{eq: Euclidean norm on the curvature}
 |F_A(p)|^2 := \sum_{1\leq i < j\leq 4} |F_{ij}(p)|^2.
\end{equation}
For a subset $U\subset X$ we denote by $\norm{F_A}_{L^\infty(U)}$ the
essential supremum of $|F_A(p)|$ over $p\in U$.

For a non-negative number $d$ we define $\moduli_d$ as the space of the 
gauge equivalence classes of ASD connections $A$
on $E$ satisfying 
\[ \norm{F_A}_{\mathrm{op}} \leq d.\]
This space is endowed with the topology of $C^\infty$-convergence over compact subsets:
A sequence $[A_n]$ converges to $[A]$ in $\moduli_d$ if and only if 
there exists a sequence of gauge transformations $g_n:E\to E$ such that $g_n(A_n)$ converges to $A$ 
in $C^\infty$ over every compact subset.
From the Uhlenbeck compactness (\cite{Uhlenbeck, Wehrheim}), the space $\moduli_d$ is compact and metrizable.
The above condition $\norm{F_A}_{\mathrm{op}}\leq d$ is motivated by the notion of Brody curves 
(Brody \cite{Brody}) in Nevanlinna theory.
Note that the norm $\norm{F_A}_{\mathrm{op}}$ does not dominate the $L^2$-norm of $F_A$.
So the $L^2$-norm of the curvature of $[A]\in \moduli_d$ can be infinite.

The space
$\moduli_d$ becomes a dynamical system with respect to the following natural $\mathbb{R}$-action:
$\mathbb{R}$ acts on $X=\mathbb{R}\times S^3$ by $s(t,\theta) := (t+s,\theta)$.
This action is lifted to the action on $E=X\times SU(2)$ by 
$s(t,\theta,u) := (t+s,\theta,u)$.
The group $\mathbb{R}$ continuously acts on $\moduli_d$ by $s[A]:=[s^*(A)]$ where $s^*(A)$ is the pull-back of $A$ 
by $s:E\to E$.
The main subject of this paper is the study of the dynamical system $\moduli_d$.
Let's start with the following example:
\begin{example} \label{example: BPST instanton}
If $d<1$ then $\moduli_d$ is equal to the one-point space.
The only one element is the gauge equivalence class of the trivial flat connection.
This fact is proved in \cite{Tsukamoto sharp lower bound}.
(The threshold value $d=1$ is different from the value given in \cite{Tsukamoto sharp lower bound}.
This is because the norm on $su(2)$ in the present paper is different from the norm in \cite{Tsukamoto sharp lower bound}
by the multiple factor $\sqrt{2}$.)

If $d=1$ then the space $\moduli_1$ contains a non-trivial element:
We define an $SU(2)$ ASD connection $A$ over the Euclidean space $\mathbb{R}^4$ by 
(\textbf{BPST instanton} \cite{BPST}) 
\begin{equation*}
   \begin{split} 
   A(x) := \frac{1}{1+|x|^2} \Bigg\{ &\begin{pmatrix}\sqrt{-1} & 0 \\ 0 & -\sqrt{-1}\end{pmatrix}
    (x_1dx_2-x_2dx_1-x_3dx_4+x_4dx_3) \\
     &+ \begin{pmatrix} 0 & 1 \\ -1 & 0 \end{pmatrix} (x_1dx_3-x_3dx_1+x_2dx_4-x_4dx_2) \\
     &+ \begin{pmatrix} 0 & \sqrt{-1} \\ \sqrt{-1} & 0 \end{pmatrix} (x_1dx_4-x_4dx_1-x_2dx_3+x_3dx_2) \Bigg\}.
   \end{split}
\end{equation*}
Let $I$ be the pull-back of $A$ by the conformal map 
\[ \mathbb{R}\times S^3\to \mathbb{R}^4\setminus \{0\}, \quad (t,\theta)\mapsto e^t\theta.\]
Then $I$ is an ASD connection on $E =X\times SU(2)$ with 
\[ |F_{I}(t,\theta)|_{\mathrm{op}} = \frac{4}{(e^t+e^{-t})^2},\quad \norm{F_I}_{\mathrm{op}} = 1.\]
Hence $[I]$ is contained in $\moduli_1$.
Therefore $\moduli_1$ contains a flat connection and the $\mathbb{R}$-orbit of $[I]$.
The authors don't know whether it contains other elements or not.
\end{example}
Therefore $\moduli_d$ is trivial for $d<1$, and possibly a simple space for $d=1$.
On the other hand
we will see later that the space $\moduli_d$ is infinite dimensional for $d>1$
(Remark \ref{remark: moduli_d is infinite dimensional for d>1}).
Moreover its topological entropy (as a topological dynamical system) is also infinite.
So $\moduli_d$ $(d>1)$ is a very large dynamical system.
A good invariant for the study of this kind of huge dynamical systems is 
\textbf{mean dimension} introduced by Gromov \cite{Gromov}.
But our present technology is a little inadequate for the study of the mean dimension of $\moduli_d$.
So we study the \textbf{local mean dimension} of $\moduli_d$ instead of mean dimension.
Local mean dimension is a variant of mean dimension introduced by \cite{Matsuo-Tsukamoto}.
For each point $[A]\in \moduli_d$ we have the non-negative number $\dim_{[A]}(\moduli_d:\mathbb{R})$
called the local mean dimension of $\moduli_d$ at $[A]$.
We define $\dim_{loc}(\moduli_d:\mathbb{R})$ as the supremum of 
$\dim_{[A]}(\moduli_d:\mathbb{R})$ over $[A]\in \moduli_d$.
Mean dimension and local mean dimension are topological invariants of dynamical systems which 
count ``dimension averaged by a group action'' in certain ways.
We review their definitions 
in Section \ref{section: review of mean dimension and local mean dimension}.

Let $A$ be a connection on $E$. We define the \textbf{energy density} $\rho(A)$ by 
\[\rho(A) := \lim_{T\to +\infty}\left(\frac{1}{8\pi^2 T}\sup_{t\in \mathbb{R}}\int_{(t,t+T)\times S^3}|F_A|^2d\vol\right).\]
Here $|F_A|$ is the Euclidean norm defined in (\ref{eq: Euclidean norm on the curvature}).
This limit always exists because we have the natural sub-additivity:
\[ \sup_{t\in \mathbb{R}}\int_{(t,t+T_1+T_2)\times S^3}|F_A|^2d\vol \leq 
  \sup_{t\in \mathbb{R}}\int_{(t,t+T_1)\times S^3}|F_A|^2d\vol +
  \sup_{t\in \mathbb{R}}\int_{(t,t+T_2)\times S^3}|F_A|^2d\vol.\]
The energy density $\rho(A)$ was first introduced in \cite{Matsuo-Tsukamoto}.
$\rho(A)$ is zero for finite energy ASD connections.
So it becomes meaningful only for infinite energy ones.
$\rho(A)$ can be seen as an ``averaged'' instanton number of $A$.
We define $\rho(d)$ as the supremum of $\rho(A)$ over all $[A]\in \moduli_d$.
$\rho(d)$ is a non-decreasing function in $d$.
It is zero for $d<1$ (Example \ref{example: BPST instanton}).
We will see later that $\rho(d)$ is positive for $d>1$ (Remark \ref{remark: moduli_d is infinite dimensional for d>1})
and that it goes to infinity as $d\to \infty$
(Example \ref{example: periodic ASD connection}).

Let $\mathcal{D}\subset [0,+\infty)$ be the set of left-discontinuous points of $\rho(d)$:
\[ \mathcal{D} = \{d\in [0,+\infty)|\, \lim_{\varepsilon\to +0}\rho(d-\varepsilon) \neq \rho(d)\}.\]
Since $\rho$ is monotone, the set $\mathcal{D}$ is at most countable.
(Indeed we don't know whether it is empty or not.)
Our main result is the following theorem.
\begin{theorem} \label{thm: main theorem}
For any $d\in [0,+\infty)\setminus \mathcal{D}$
\[ \dim_{loc}(\moduli_d:\mathbb{R}) = 8\rho(d).\]
\end{theorem}
Since $\mathcal{D}$ is at most countable,
we get the formula of the local mean dimension of $\moduli_d$ for almost every $d\geq 0$.
\begin{remark}
Some readers might feel that the operator norm $\norm{F_A}_{\mathrm{op}}$ used in the definition of $\moduli_d$
seems strange.
Indeed this choice leads us to a very satisfactory result.
But we will briefly discuss another possibility in Appendix.
\end{remark}

\subsection{Non-degenerate ASD connections} \label{subsection: non-degenerate ASD connections}

The following notion is very important in the argument of the paper:
\begin{definition}
Let $[A]\in \moduli_d$ $(d\geq 0)$.
$A$ is said to be \textbf{non-degenerate} if the closure of the $\mathbb{R}$-orbit of $[A]$ in $\moduli_d$ does not contain 
a gauge equivalence class of a flat connection.
\end{definition}
This definition is motivated by the classical work of Yosida \cite{Yosida}
in complex analysis.
Yosida studied a similar non-degeneracy condition for meromorphic functions 
$f:\mathbb{C}\to \mathbb{C}P^1$. (He used the terminology ``meromorphic functions of 
first category''.)
Eremenko \cite[Section 4]{Eremenko} discussed it for holomorphic curves $f:\mathbb{C}\to \mathbb{C}P^N$,
and Gromov \cite[p. 399]{Gromov} studied a similar condition for more general holomorphic maps.

\begin{example}
Let $A$ be an \textbf{instanton}, i.e. an ASD connection on $E$ with finite energy
\[ \int_{X} |F_A|^2d\vol < +\infty.\]
Then $s[A]$ converges to a gauge equivalence class of a flat connection when $s\to \pm \infty$
(Donaldson \cite[Chapter 4, Proposition 4.3]{Donaldson}).
So $A$ is a degenerate (i.e. not non-degenerate) ASD connection.
\end{example}
\begin{example} \label{example: periodic ASD connection}
An ASD connection 
$A$ on $E$ is said to be \textbf{periodic} (\cite{Matsuo-Tsukamoto}) if there exist $T>0$, 
a principal $SU(2)$ bundle $F$ over $(\mathbb{R}/T\mathbb{Z})\times S^3$
and an ASD connection $B$ on $F$ such that $(E,A)$ is isomorphic to the pull-back $(\pi^*(F),\pi^*(B))$.
Here $\pi: \mathbb{R}\times S^3 \to (\mathbb{R}/T\mathbb{Z})\times S^3$ is the natural projection.
If $A$ is periodic, then the energy density $\rho(A)$ is given by 
\[ \rho(A) = c_2(F)/T.\]
If $A$ is periodic and non-flat, then $A$ is non-degenerate.
By Taubes \cite{Taubes}, every principal $SU(2)$ bundle $F$ on $(\mathbb{R}/T\mathbb{Z})\times S^3$ with 
$c_2(F)\geq 0$ admits an ASD connection.
Therefore we have a lot of periodic ASD connections.
From this fact we can easily see that the function $\rho(d)$ introduced in the previous subsection goes to 
infinity as $d\to \infty$.
\end{example}
\begin{lemma} \label{lemma: equivalent condition of non-degeneracy}
Let $[A]\in \moduli_d$.
$A$ is non-degenerate if and only if there exist $\delta>0$ and $T>0$ such that 
for any interval $(\alpha,\beta)\subset \mathbb{R}$
of length $T$ we have 
\begin{equation} \label{eq: non-degeneracy}
  \norm{F_A}_{L^\infty((\alpha,\beta)\times S^3)} \geq \delta.
\end{equation}
\end{lemma}
\begin{proof}
This is a Yang--Mills analogue of the result of Yosida \cite[Theorem 4]{Yosida}.
Suppose that $A$ does not satisfy (\ref{eq: non-degeneracy}) for $T=1$.
Then there exist $\{\alpha_n\}_{n\geq 1} \subset \mathbb{R}$ such that 
$\norm{F_A}_{L^\infty((\alpha_n,\alpha_n+1)\times S^3)} < 1/n$.
By choosing a subsequence we can assume that $\alpha_n[A]$ converges to $[B]$ in $\moduli_d$.
Then $F_B=0$ over $(0,1)\times S^3$.
By the unique continuation, $F_B=0$ all over $X$.
Hence $B$ is flat and $A$ is degenerate.

Suppose the above condition (\ref{eq: non-degeneracy}) holds for some $\delta>0$ and $T>0$.
Then any element $[B]$ in the closure of the $\mathbb{R}$-orbit of $[A]$ satisfies 
$\norm{F_B}_{L^\infty((\alpha,\beta)\times S^3)}\geq \delta$ for every interval $(\alpha,\beta)\subset \mathbb{R}$
of length $T$. Hence $B$ is not flat.

Note that the above argument also proves the following:
$[A]$ is non-degenerate if and only if for any $T>0$ there exists $\delta>0$ such that for 
any interval $(\alpha,\beta)\subset \mathbb{R}$ of length $T$ we have 
$\norm{F_A}_{L^\infty((\alpha,\beta)\times S^3)}\geq \delta$.
\end{proof}
\begin{remark} \label{remark: energy density of non-degenerate ASD connection}
By the same argument we can prove the following:
$[A]\in \moduli_d$ is non-degenerate if and only if there exist $\delta>0$ and $T>0$ such that 
for any interval $(\alpha,\beta)\subset \mathbb{R}$ of length $T$ we have 
\[ \norm{F_A}_{L^2((\alpha,\beta)\times S^3)}\geq \delta.\]
In particular if $[A]\in \moduli_d$ is non-degenerate then its energy density $\rho(A)$ is positive.
\end{remark}

The following Theorem \ref{thm: upper bound} is proved in \cite[Theorem 1.2]{Matsuo-Tsukamoto}.
(The paper \cite{Matsuo-Tsukamoto} adopts a little different setting.
So we explain how to deduce this result from \cite{Matsuo-Tsukamoto} in Appendix.)
\begin{theorem} \label{thm: upper bound}
For any $[A]\in \moduli_d$, 
\[ \dim_{[A]}(\moduli_d:\mathbb{R}) \leq 8\rho(A).\]
Hence 
\[ \dim_{loc}(\moduli_d:\mathbb{R}) = \sup_{[A]\in \moduli_d}\dim_{[A]}(\moduli_d:\mathbb{R}) \leq 8\rho(d).\]
\end{theorem}

The lower bound on the local mean dimension is given by using the next two theorems.
\begin{theorem} \label{thm: local mean dimension around non-degenerate ASD connection}
Let $A$ be a non-degenerate ASD connection on $E$ with $\norm{F_A}_{\mathrm{op}} < d$.
Then 
\[ \dim_{[A]}(\moduli_d:\mathbb{R}) = 8\rho(A).\]
\end{theorem}

In this theorem the strict inequality condition $\norm{F_A}_{\mathrm{op}} < d$ is purely technical.
The point is the non-degeneracy assumption.
This makes the situation simpler.
It is more difficult to study the local structure of $\moduli_d$ around degenerate ASD connections.
We postpone it to a future paper.
In the present paper we bypass it by using the following theorem.
\begin{theorem} \label{thm: abundance of non-degenerate ASD connections}
Suppose $d>1$, and
let $A$ be an ASD connection on $E$ with $\norm{F_A}_{\mathrm{op}} < d$.
For any $\varepsilon>0$ there exists a non-degenerate ASD connection $\tilde{A}$ on $E$ satisfying 
\[ \norm{F(\tilde{A})}_{\mathrm{op}} < d, \quad \rho(\tilde{A}) > \rho(A) -\varepsilon.\]
\end{theorem}
Roughly speaking, this theorem means that we can replace a degenerate ASD connection by a non-degenerate one without 
losing energy.
In the above statement we supposed $d>1$ because there does not exist a non-flat ASD connection $A$ on $E$ 
satisfying $\norm{F_A}_{\mathrm{op}}<1$ (Example \ref{example: BPST instanton}).

The main task of the paper is to prove Theorems \ref{thm: local mean dimension around non-degenerate ASD connection}
and \ref{thm: abundance of non-degenerate ASD connections}.
Here we prove the main theorem by assuming them:

\begin{proof}[Proof of Theorem \ref{thm: main theorem} (assuming Theorems 
\ref{thm: local mean dimension around non-degenerate ASD connection} and 
\ref{thm: abundance of non-degenerate ASD connections})]
We always have the upper bound $\dim_{loc}(\moduli_d:\mathbb{R}) \leq 8\rho(d)$
by Theorem \ref{thm: upper bound}.
So the problem is the lower bound.

Let $\rho_0(d)$ be the supremum of $\rho(A)$ over $[A]\in \moduli_d$ satisfying
$\norm{F_A}_{\mathrm{op}} < d$.
Obviously $\rho_0(d)\leq \rho(d)$.
Then 
\begin{equation}  \label{eq: lower bound on local mean dimension by rho_0}
 \dim_{loc}(\moduli_d:\mathbb{R})\geq 8 \rho_0(d).
\end{equation}
This is proved as follows:

(Case 1)
Suppose $d\leq 1$. Then the condition $\norm{F_A}_{\mathrm{op}} < d$ implies $F_A\equiv 0$.
(See Example \ref{example: BPST instanton}.)
Hence $\rho_0(d)=0$ and the above (\ref{eq: lower bound on local mean dimension by rho_0})
trivially holds.
 
(Case 2) Suppose $d >1$.
Take $[A]\in \moduli_d$ with $\norm{F_A}_{\mathrm{op}} < d$.
For any $\varepsilon>0$ there exists a non-degenerate ASD connection $\tilde{A}$ on $E$ satisfying 
$\norm{F(\tilde{A})}_{\mathrm{op}}<d$ and $\rho(\tilde{A})> \rho(A)-\varepsilon$ 
(Theorem \ref{thm: abundance of non-degenerate ASD connections}).
By applying Theorem \ref{thm: local mean dimension around non-degenerate ASD connection} to $\tilde{A}$
\[ \dim_{loc}(\moduli_d:\mathbb{R}) \geq \dim_{[\tilde{A}]}(\moduli_d:\mathbb{R})
   = 8\rho(\tilde{A}) > 8(\rho(A)-\varepsilon).\]
Since $\varepsilon>0$ is arbitrary, $\dim_{loc}(\moduli_d:\mathbb{R})\geq 8\rho(A)$.
Taking the supremum over $A$, we get the above (\ref{eq: lower bound on local mean dimension by rho_0}).

For any $\varepsilon>0$, we have $\rho(d-\varepsilon)\leq \rho_0(d) \leq \rho(d)$.
Hence if $\rho$ is left-continuous at $d$ (i.e. $d\not\in \mathcal{D}$), then we have 
$\rho_0(d) =\rho(d)$.
Therefore
\[ \dim_{loc}(\moduli_d:\mathbb{R})\geq 8\rho(d) \quad (d\in [0,+\infty)\setminus \mathcal{D}).\]
\end{proof}

\begin{remark}  \label{remark: moduli_d is infinite dimensional for d>1}
Let $d>1$. By applying Theorem \ref{thm: abundance of non-degenerate ASD connections}
to a flat connection, we can conclude that $\moduli_d$ always contains a non-degenerate ASD connection.
(Indeed $\moduli_d$ always contains a non-flat periodic ASD connection.
See Remark \ref{remark: final remark}.)
Since the energy density of a non-degenerate ASD connection is positive 
(Remark \ref{remark: energy density of non-degenerate ASD connection}),
the function $\rho(d)$ is positive for $d>1$.
Moreover by Theorem \ref{thm: local mean dimension around non-degenerate ASD connection},
the local mean dimension of $\moduli_d$ is also positive for $d>1$.
In particular $\moduli_d$ is infinite dimensional for $d>1$.
\end{remark}

\subsection{Ideas of the proofs} \label{subsection: ideas of the proofs}

We explain the ideas of the proofs of Theorems 
\ref{thm: local mean dimension around non-degenerate ASD connection} and 
\ref{thm: abundance of non-degenerate ASD connections}.

The basic idea of the proof of Theorem \ref{thm: local mean dimension around non-degenerate ASD connection}
is a deformation theory.
Let $A$ be a non-degenerate ASD connection on $E$ satisfying $\norm{F_A}_{\mathrm{op}} < d$.
Let $H^1_A$ be the Banach space of $a\in \Omega^1(\ad E)$ satisfying 
\[ d_A^* a = d_A^+a = 0, \quad \norm{a}_{L^\infty(X)} < \infty.\]
Here $d_A^*$ is the formal adjoint of $d_A:\Omega^0(\ad E)\to \Omega^1(\ad E)$, 
and $d_A^+$ is the self-adjoint part of $d_A:\Omega^1(\ad E)\to \Omega^2(\ad E)$.
For each $a\in H^1_A$ the connection $A+a$ is almost ASD:
$F^+(A+a) = O(a^2)$.
Therefore there exists a small $R>0$ such that for each $a\in B_R(H^1_A)$ (the $R$-ball with respect to 
$\norm{\cdot}_{L^\infty(X)}$) we can construct a small perturbation $a'$ of $a$ satisfying $F^+(A+a')=0$.  
So we get a deformation map:
\begin{equation} \label{eq: deformation map in introduction}
 B_R(H^1_A)\to \moduli_d, \quad a\mapsto [A+a'].
\end{equation}
We study the local mean dimension of $\moduli_d$ through this map.

A construction of the map (\ref{eq: deformation map in introduction}) does not require the non-degeneracy condition
of $A$.
But a further study of (\ref{eq: deformation map in introduction}) requires it.
We need to compare the distances of the both sides of (\ref{eq: deformation map in introduction}).
$\moduli_d$ is a quotient space by gauge transformations.
Hence its metric structure is more complicated than that of $B_R(H^1_A)$.
For example, even if $a,b\in B_R(H^1_A)$ are not close to each other, the points 
$[A+a']$ and $[A+b']$ might be very close to each other in $\moduli_d$.
We need the non-degeneracy condition for addressing this problem.
This is a technical issue. So here we don't go into the detail but just point out that 
the above map (\ref{eq: deformation map in introduction}) becomes injective if $R\ll 1$ and $A$ 
is non-degenerate.
(This injectivity is not enough for our main purpose.
The result we need is stated in Lemma \ref{lemma: distortion of the deformation map}, and it is
based on the study of the Coulomb gauge condition in Section \ref{section: Coulomb gauge}.)

Assume that we have a good understanding of the deformation map (\ref{eq: deformation map in introduction}).
A next problem is the study of the Banach space $H^1_A$.
We investigate a structure of finite dimensional linear subspaces of $H^1_A$.
($H^1_A$ itself is infinite dimensional.)
We need the following result (Proposition \ref{prop: main result of the study of H^1_A}):
For any interval $(\alpha,\beta)\subset \mathbb{R}$ of length $>2$ there exists a finite dimensional linear subspace
$V\subset H^1_A$ such that 
\begin{equation*}
  \begin{split}
  \dim V \geq \frac{1}{\pi^2}\int_{(\alpha,\beta)\times S^3}|F_A|^2d\vol -\const_A,\\
  \forall a\in H^1_A: \, \norm{a}_{L^\infty(X)} \leq 2\norm{a}_{L^\infty((\alpha,\beta)\times S^3)}.
  \end{split} 
\end{equation*}
The energy density $\rho(A)$ comes into our argument through the first condition of $V$.
The second condition means that essentially all the information of $a\in V$ is contained in the region 
$(\alpha,\beta)\times S^3$.
A main ingredient of the proof of this result is the Atiyah--Singer index theorem.
Combining this knowledge on $H^1_A$ with the study of the deformation map 
(\ref{eq: deformation map in introduction}), 
we can prove Theorem \ref{thm: local mean dimension around non-degenerate ASD connection}.
The proof is finished in Section \ref{subsection: proof of Theorem local mean dimension around A}.

Next we explain the idea of the proof of Theorem \ref{thm: abundance of non-degenerate ASD connections}.
Suppose $d>1$ and that $A$ is a degenerate ASD connection on $E$ with $\norm{F_A}_{\mathrm{op}} < d$.
We want to replace $A$ with a non-degenerate one.
The idea is gluing instantons.
Lemma \ref{lemma: equivalent condition of non-degeneracy} implies that $A$ has a region
where the curvature $F_A$ is very small.
We glue an instanton $I$ (described in Example \ref{example: BPST instanton}) to $A$ over such a ``degenerate region''.
$A$ has infinitely many degenerate regions.
So we need to glue infinitely many instantons to $A$.

More precisely the argument goes as follows:
Let $0< \delta\ll 1$ and $T\gg 1$.
We define $J\subset \mathbb{Z}$ as the set of $n\in \mathbb{Z}$ such that 
$|F_A| < \delta$ over $[nT,(n+1)T]\times S^3$.
Since $A$ is degenerate, the set $J$ is infinite.
For each $n\in J$ we glue (an appropriate translation of) the instanton $I$ to $A$ over 
the region $[nT,(n+1)T]\times S^3$.
If we choose $\delta$ sufficiently small and $T$ sufficiently large, then 
the resulting new ASD connection $\tilde{A}$ becomes non-degenerate and satisfies
$\norm{F(\tilde{A})}_{\mathrm{op}} < d$.
Moreover, roughly speaking, gluing instantons increases the energy of connections.
So we have $\rho(\tilde{A}) > \rho(A)-\varepsilon$.

The paper \cite{Matsuo-Tsukamoto Brody curve} is the origin of our idea to use the deformation theory of non-degenerate 
objects and gluing infinitely many instantons.
In \cite{Matsuo-Tsukamoto Brody curve} we study the mean dimension of the system of Brody curves 
(holomorphic $1$-Lipschitz maps) $f:\mathbb{C}\to \mathbb{C}P^N$
by developing the deformation theory of non-degenerate Brody curves and gluing technique of infinitely many rational curves.
After the authors wrote the paper \cite{Matsuo-Tsukamoto Brody curve}, they felt that the ideas of 
\cite{Matsuo-Tsukamoto Brody curve} have a wide applicability beyond the holomorphic curve theory.
The second main purpose of the present paper is to show that a basic structure of the argument in
\cite{Matsuo-Tsukamoto Brody curve} is certainly flexible and can be also applied to Yang--Mills theory.
The authors are satisfied with the result.

The main difference between the case of Brody curves and Yang--Mills theory is the presence of gauge transformations. 
A substantial part of the present paper is devoted to the study of the method to deal with gauge transformations.
(The technique of perturbing Hermitian metrics described in \cite[Section 4.2]{Matsuo-Tsukamoto Brody curve}
might have a flavor of gauge fixing. But it is much simpler.)
At least for our present technology, the Yang--Mills case is more involved than Brody curves.
Another, relatively minor, difference is the techniques of gluing.
The gluing construction in \cite{Matsuo-Tsukamoto Brody curve} is more elementary than that of the present paper.
The reason is that for meromorphic functions $f$ and $g$ in $\mathbb{C}$ we have a natural definition of
their sum $f+g$.
But we don't have such a definition for the ``sum'' of ASD connections.

\subsection{Organization of the paper} 

Section \ref{section: review of mean dimension and local mean dimension}
is a review of mean dimension and local mean dimension.
Section \ref{section: Coulomb gauge} is devoted to the study of the Coulomb gauge condition.
In Section \ref{section: parameter space of the deformation}
we study the Banach space $H^1_A$.
In Section \ref{section: deformation theory} we develop the deformation theory of non-degenerate 
ASD connections and prove Theorem \ref{thm: local mean dimension around non-degenerate ASD connection}.
In Section \ref{section: gluing infinitely many instantons}
we study the gluing method and prove Theorem \ref{thm: abundance of non-degenerate ASD connections}.
In Appendix we investigate another definition of the ASD moduli space.

\subsection{Notations} \label{subsection: notations}

$\bullet$
In most of the argument the variable $t$ means the natural projection $t:\mathbb{R}\times S^3\to \mathbb{R}$.

$\bullet$ 
The value of $d$ (which is used to define $\moduli_d$) is fixed in the rest of this paper (except for Appendix).
So we usually omit to write the dependence on $d$.
We adopt the following notation:
\begin{notation}
For two quantities $x$ and $y$ we write
\[ x \lesssim y\]
if there exists a positive constant $C(d)$ which depends only on $d$ such that 
$x \leq C(d) y$.
Let $A$ be a connection on $E$. We also use the following notation:
\[ x \lesssim_A y.\]
This means that there exists a positive constant $C(d,A)$ which depends only on $d$ and $A$ such that 
$x\leq C(d,A) y$.
The notation $x\lesssim_A y$ is used in Sections \ref{section: Coulomb gauge}, \ref{section: parameter space of the deformation}
and \ref{section: deformation theory} where we fix a connection $A$ in most of the argument.
\end{notation}

$\bullet$
Let $A$ be a connection on $E$. Let $k\geq 0$ be an integer, and let $p\geq 1$.
For $\xi\in \Omega^i(\ad E)$ $(0\leq i\leq 4)$
and a subset $U\subset X$, we define a norm $\norm{\xi}_{L^p_{k,A}(U)}$ by 
\[ \norm{\xi}_{L^p_{k,A}(U)} := \left(\sum_{j=0}^k \norm{\nabla_A^j\xi}_{L^p(U)}^p\right)^{1/p}.\]
For $\alpha<\beta$ we often denote the norm $\norm{\xi}_{L^p_{k,A}((\alpha,\beta)\times S^3)}$
by $\norm{\xi}_{L^p_{k,A}(\alpha<t<\beta)}$.

\section{Review of mean dimension and local mean dimension} 
\label{section: review of mean dimension and local mean dimension}
In this section we review mean dimension and local mean dimension.
Mean dimension was introduced by Gromov \cite{Gromov}.
Lindenstrauss--Weiss \cite{Lindenstrauss-Weiss} and Lindenstrauss \cite{Lindenstrauss} 
also gave fundamental contributions to the basics of this invariant.
Local mean dimension was introduced in \cite{Matsuo-Tsukamoto}.

Let $(M,\dist)$ be a compact metric space
($\dist$ is a distance function on $M$).
Let $N$ be a topological space, and let $f:M\to N$ be a continuous map.
For $\varepsilon>0$, $f$ is called an $\varepsilon$-embedding if $\Diam f^{-1}(y)\leq \varepsilon$
for all $y\in N$.
Here $\Diam f^{-1}(y)$ is the supremum of $\dist(x_1, x_2)$ over all $x_1$ and $x_2$ in the fiber $f^{-1}(y)$.
Let $\widim_\varepsilon(M,\dist)$ be the minimum integer $n\geq 0$ such that 
there exist an $n$-dimensional polyhedron $P$ and an $\varepsilon$-embedding $f:M\to P$.
The topological dimension $\dim M$ is equal to the limit of $\widim_\varepsilon(M,\dist)$ as $\varepsilon\to 0$.

The following important example was given in \cite[p. 333]{Gromov}.
This will be used in Section \ref{section: deformation theory}.
The detailed proofs are given in Gournay \cite[Lemma 2.5]{Gournay widths} and Tsukamoto
\cite[Appendix]{Tsukamoto-deformation}.
\begin{example} \label{example: widim of the Banach ball}
Let $(V,\norm{\cdot})$ be a finite dimensional Banach space.
Let $B_r(V)$ be the closed ball of radius $r>0$ centered at the origin.
Then 
\[ \widim_\varepsilon(B_r(V), \norm{\cdot}) = \dim V, \quad (0<\varepsilon <r).\]
\end{example}
Suppose that the Lie group $\mathbb{R}$ continuously acts on a compact metric space $(M,\dist)$.
For a subset $\Omega\subset \mathbb{R}$, we define a new distance $\dist_\Omega$ on $M$ by 
$\dist_\Omega(x,y) := \sup_{a\in \Omega} \dist(a.x,a.y)$ $(x,y\in M)$.
We define the mean dimension $\dim(M:\mathbb{R})$ by 
\[ \dim(M:\mathbb{R}) := \lim_{\varepsilon\to 0}
\left(\lim_{T\to +\infty}\frac{\widim_\varepsilon(M,\dist_{(0,T)})}{T}\right).\]
This limit always exists because we have the following sub-additivity:
\[ \widim_\varepsilon(M, \dist_{(0,T_1+T_2)}) \leq 
   \widim_\varepsilon(M,\dist_{(0,T_1)})+\widim_\varepsilon(M,\dist_{(0,T_2)}).\]
The mean dimension $\dim(M:\mathbb{R})$ is a topological invariant.
(This means that its value is independent of the choice of a distance function compatible with the topology.)
If $M$ is finite dimensional, then the mean dimension $\dim(M:\mathbb{R})$ is equal to $0$.

Let $N\subset M$ be a closed subset.
The function 
\[ T\mapsto \sup_{a\in \mathbb{R}}\widim_\varepsilon(N,\dist_{(a,a+T)}) \]
is also sub-additive.
So we can define the following quantity:
\[ \dim(N:\mathbb{R}):=\lim_{\varepsilon\to 0}
  \left(\lim_{T\to +\infty}\frac{\sup_{a\in \mathbb{R}}\widim_\varepsilon(N, \dist_{(a,a+T)})}{T}\right).\]

For $r>0$ and $p\in M$ we define $B_r(p)_{\mathbb{R}}$ as the set of points $x\in M$ 
satisfying $\dist_{\mathbb{R}}(p,x)\leq r$.
(Note that $\dist_{\mathbb{R}}(p,x)\leq r$ means $\dist(a.p,a.x)\leq r$ for all $a\in \mathbb{R}$.)
We define the local mean dimension $\dim_{p}(M:\mathbb{R})$ at $p$ by 
\[ \dim_p(M:\mathbb{R}) := \lim_{r\to 0}\dim(B_r(p)_{\mathbb{R}}:\mathbb{R}).\]
We define the local mean dimension $\dim_{loc}(M:\mathbb{R})$ by 
\[ \dim_{loc}(M:\mathbb{R}) := \sup_{p\in M}\dim_p(M:\mathbb{R}).\]
$\dim_p(M:\mathbb{R})$ and $\dim_{loc}(M:\mathbb{R})$ are topological invariants of the dynamical system $M$.
We always have 
\[ \dim_p(M:\mathbb{R}) \leq \dim_{loc}(M:\mathbb{R}) \leq \dim(M:\mathbb{R}).\]

In this paper we define mean dimension only for $\mathbb{R}$-actions.
But we can define it for more general group actions.
Gromov \cite{Gromov} defined mean dimension for actions of amenable groups.
The most basic example is the natural $\mathbb{Z}$-action (shift action) on the infinite dimensional cube 
\[ [0,1]^{\mathbb{Z}} := \cdots\times [0,1]\times [0,1]\times [0,1]\times \cdots.\]
Its mean dimension and local mean dimension are given by 
\[ \dim_{0}([0,1]^{\mathbb{Z}}:\mathbb{Z}) = \dim_{loc}([0,1]^{\mathbb{Z}}:\mathbb{Z}) 
  = \dim([0,1]^{\mathbb{Z}}:\mathbb{Z}) = 1.\]
Here $0 = (\dots,0,0,0,\dots)\in [0,1]^{\mathbb{Z}}$.
We don't need this result in this paper. So we omit the detail.
The detailed explanations can be found in 
Lindenstrauss--Weiss \cite[Proposition 3.3]{Lindenstrauss-Weiss} and \cite[Example 2.9]{Matsuo-Tsukamoto}.

\section{Coulomb gauge} \label{section: Coulomb gauge}

In this section we study a gauge fixing condition.
This is a technical step toward the proof
of Theorem \ref{thm: local mean dimension around non-degenerate ASD connection}.
The ASD equation is not elliptic and admits a large symmetry of gauge transformations.
So in the standard Yang--Mills theory 
we introduce the Coulomb gauge condition in order to break the gauge symmetry and 
get the ellipticity of the equation.
In our situation the gauge fixing seems more involved than in the standard argument.
A difficulty lies in the point that we need to consider \textit{all} gauge transformations 
$g:E\to E$ (without any asymptotic condition at the end) and that they don't form a Banach Lie group.
The main result of this section is Proposition \ref{prop: Coulomb gauge}.
But its statement is not simple.
Probably Corollary \ref{cor: gauge fixing} is easier to understand.
So it might be helpful for some readers to look at Corollary \ref{cor: gauge fixing}
before reading the proof of Proposition \ref{prop: Coulomb gauge}.

The next lemma is proved in \cite[Corollary 6.3]{Matsuo-Tsukamoto}.
This is crucial for our argument.
\begin{lemma} \label{lemma: non-flat implies irreducible}
If $A$ is a non-flat ASD connection on $E$ satisfying $\norm{F_A}_{\mathrm{op}} < \infty$, 
then $A$ is irreducible.
(Recall that $A$ is said to be reducible if there is a gauge transformation $g \neq \pm 1$ satisfying 
$g(A)=A$.
$A$ is said to be irreducible if $A$ is not reducible.)
\end{lemma}

In the rest of this section we always suppose that $A$ is a non-degenerate ASD connection on $E$
satisfying $\norm{F_A}_{\mathrm{op}}\leq d$.
The next lemma shows crucial properties of non-degenerate ASD connections.
\begin{lemma} \label{lemma: basic estimates from non-degeneracy}
\noindent 
(i) For any $s\in \mathbb{R}$ and any $u\in \Omega^0(\ad E)$,
\[ \int_{s < t< s+1}|u|^2d\vol \leq C_1(A)\int_{s<t<s+1}|d_Au|^2d\vol.\]

\noindent 
(ii) For any $s\in \mathbb{R}$ and any gauge transformation $g:E\to E$,
\[ \min\left(\norm{g-1}_{L^\infty(s<t<s+1)},\norm{g+1}_{L^\infty(s<t<s+1)}\right) 
   \leq C_2(A) \norm{d_A g}_{L^\infty(s<t<s+1)}.\]
We will abbreviate the left-hand-side to $\min_{\pm} \norm{g\pm 1}_{L^\infty(s<t<s+1)}$. 
\end{lemma}
\begin{proof}
(i)
Suppose that the statement is false.
Then there exist $s_n \in \mathbb{R}$ and $u_n\in \Omega^0(\ad E)$ satisfying 
\[ 1 = \int_{s_n<t<s_n+1}|u_n|^2d\vol > n\int_{s_n<t<s_n+1}|d_A u_n|^2 d\vol.\]
Set $v_n := s_n^*(u_n)$ and $A_n:=s_n^*(A)$ (the pull-backs by $s_n:E\to E$).
Then 
\[ 1 = \int_{0<t<1}|v_n|^2 d\vol > n\int_{0<t<1}|d_{A_n} v_n|^2d\vol.\]
Since $\moduli_d$ is compact, there exist a sequence of natural numbers 
$n_1<n_2<n_3<\cdots$ and gauge transformations $g_k :E\to E$ $(k\geq 1)$ such that 
$B_k:=g_k(A_{n_k})$ converges to some $B$ in $\mathcal{C}^\infty$ over every compact subset of $X$.
Since $A$ is non-degenerate, $B$ is not flat and hence irreducible by Lemma \ref{lemma: non-flat implies irreducible}.
Set $w_k:=g_k(v_{n_k})$. Then 
\[ 1 = \int_{0<t<1}|w_k|^2d\vol > n_k \int_{0<t<1} |d_{B_k}w_k|^2d\vol.\]
Since $d_B w_k = d_{B_k}w_k + [B-B_k, w_k]$, the sequence $\{w_k\}$ is bounded in $L^2_{1,B}((0,1)\times S^3)$.
Hence, by choosing a subsequence, we can assume that $w_k$ weakly converges to 
some $w$ in $L^2_{1,B}((0,1)\times S^3)$.
We have $\norm{w}_{L^2(0<t<1)}=1$ and $d_B w=0$ over $(0,1)\times S^3$.
This means that the connection $B$ is reducible over $(0,1)\times S^3$.
By the unique continuation theorem \cite[p. 150]{Donaldson-Kronheimer}, $B$ is reducible over $X$.
This is a contradiction.

(ii) Fix $4<p<\infty$. (Note that the Sobolev embedding $L^p_1\hookrightarrow C^0$ is compact.)
By an argument similar to the above (i),
we can prove the following statement:
For any $s \in \mathbb{R}$ and any $u\in \Omega^0(\ad E)$ 
\begin{equation} \label{eq: auxiliary estimate for basic estimates from non-degeneracy}
  \norm{u}_{L^\infty(s<t<s+1)} \lesssim_A C(p) \norm{d_A u}_{L^p(s<t<s+1)}.
\end{equation}

We prove (ii) by using this statement.
Suppose (ii) is false.
Then, as in the proof of (i), there exist connections $A_n$ (which are translations of $A$) and 
gauge transformations $g_n:E\to E$ satisfying 
\[ \min_{\pm}\norm{g_n\pm 1}_{L^\infty(0<t<1)} > n\norm{d_{A_n}g_n}_{L^\infty(0<t<1)}.\]
We can choose a sequence of natural numbers $n_1<n_2<n_3<\cdots$ and 
gauge transformations $h_k:E\to E$ $(k\geq 1)$ such that 
$B_k:=h_k(A_{n_k})$ converges to some $B$ in $\mathcal{C}^\infty$ over every compact subset.
$B$ is irreducible.
Set $g'_k := h_k g_{n_k} h_k^{-1}$. Then 
\begin{equation} \label{eq: |g'_k +- 1| > n_k |d_{B_k}g'_k|}
 \min_{\pm}\norm{g'_k\pm 1}_{L^\infty(0<t<1)} > n_k \norm{d_{B_k}g'_k}_{L^\infty(0<t<1)}.
\end{equation}
$\{g'_k\}$ is bounded in $L^p_{1,B}((0,1)\times S^3)$.
By choosing a subsequence, $g'_k$ converges to some $g'$ weakly in $L^p_{1,B}((0,1)\times S^3)$
and strongly in $L^\infty((0,1)\times S^3)$.
We have $d_B g'=0$.
Since $B$ is irreducible, $g'=\pm 1$.
We can assume $g'=1$ without loss of generality.
Then there are $u_k\in L^p_{1,B}((0,1)\times S^3, \Lambda^0(\ad E))$ $(k\gg 1)$ satisfying
$g'_k = e^{u_k}$ and $|u_k|\lesssim |g'_k-1|$ over $0<t<1$.
Then by (\ref{eq: auxiliary estimate for basic estimates from non-degeneracy})
\[ \norm{g'_k-1}_{L^\infty(0<t<1)}\lesssim \norm{u_k}_{L^\infty(0<t<1)}
    \lesssim_A C(p) \norm{d_{B_k} u_k}_{L^p(0<t<1)}.\]
We have $\norm{d_{B_k}u_k}_{L^p(0<t<1)}\leq 2\norm{d_{B_k}g'_k}_{L^p(0<t<1)}$
for $k\gg 1$.
Hence, for $k\gg 1$,
\[ \norm{g'_k-1}_{L^\infty(0<t<1)} \lesssim_A \norm{d_{B_k}g'_k}_{L^\infty(0<t<1)}.\]
This contradicts (\ref{eq: |g'_k +- 1| > n_k |d_{B_k}g'_k|}). 
\end{proof}

\begin{lemma} \label{lemma: estimate on gauge transformation from non-degeneracy}
There exists a positive number $\varepsilon_1 = \varepsilon_1(A)$ such that, 
for any integers $m<n$ and any gauge transformation $g:E\to E$, 
if $\norm{d_A g}_{L^\infty(m<t<n)}\leq \varepsilon_1$ then 
\[ \min_{\pm} \norm{g\pm 1}_{L^\infty(m<t<n)}
    \leq C_2(A) \norm{d_A g}_{L^\infty(m<t<n)}.\]
This is also true for the case $(m,n)=(-\infty,\infty)$.
\end{lemma}
\begin{proof}
For simplicity we suppose $m=0$.
By Lemma \ref{lemma: basic estimates from non-degeneracy} (ii), for every $k\in \mathbb{Z}$,
\begin{equation} \label{eq: |g+-1| leq C |d_A g|}
 \min_{\pm} \norm{g\pm 1}_{L^\infty(k<t<k+1)}\leq C_2
    \norm{d_Ag}_{L^\infty(k<t<k+1)}.
\end{equation}
Take a positive number $\varepsilon_1 = \varepsilon_1(A)$ satisfying 
$(C_2+1)\varepsilon_1<1$.
Suppose $\norm{d_A g}_{L^\infty(0<t<n)}\leq \varepsilon_1$.
We can also suppose 
\[ \norm{g-1}_{L^\infty(0<t<1)} \leq \norm{g+1}_{L^\infty(0<t<1)}\]
without loss of generality.
Then $\norm{g-1}_{L^\infty(0<t<1)}\leq C_2\varepsilon_1$. 
Since $|d_A g|\leq \varepsilon_1$ over $0\leq t\leq 2$, we have 
\[ \norm{g-1}_{L^\infty(1<t<2)} \leq (C_2+1)\varepsilon_1 < 1.\]
Then $\norm{g+1}_{L^\infty(1<t<2)} \geq 2-(C_2+1)\varepsilon_1 > 1$.
Hence 
\[ \norm{g-1}_{L^\infty(1<t<2)} < \norm{g+1}_{L^\infty(1<t<2)}.\]
In the same way, we can prove that for every $0 \leq k < n$
\[ \norm{g-1}_{L^\infty(k<t<k+1)} < \norm{g+1}_{L^\infty(k<t<k+1)}.\]
By (\ref{eq: |g+-1| leq C |d_A g|}),
\[ \norm{g-1}_{L^\infty(k<t<k+1)} \leq C_2
    \norm{d_Ag}_{L^\infty(k<t<k+1)}. \]
Thus $\norm{g-1}_{L^\infty(0<t<n)}\leq C_2\norm{d_A g}_{L^\infty(0<t<n)}$.
\end{proof}

Fix a positive integer $T=T(A)$ satisfying 
\begin{equation}  \label{eq: definition of T}
 \frac{10C_1 + 20\sqrt{C_1}}{T} < \frac{1}{4}.
\end{equation}
Here $C_1=C_1(A)$ is the positive constant introduced in Lemma 
\ref{lemma: basic estimates from non-degeneracy} (i).
For the later convenience (Lemma \ref{lemma: distortion of the deformation map}) we assume $T>3$.
For $\xi\in \Omega^i(\ad E)$ and integers $m \leq n$, we set 
\[ \norm{\xi}_m^n := \max_{m\leq k \leq n}\norm{\xi}_{L^2(kT<t<(k+1)T)}.\]

Let $d_A^*:\Omega^1(\ad E)\to \Omega^0(\ad E)$ be the formal adjoint of 
$d_A:\Omega^0(\ad E)\to \Omega^1(\ad E)$.
We set $\Delta_A u := d_A^*d_A u$ for $u\in \Omega^0(\ad E)$.

\begin{lemma}  \label{lemma: preliminary technical estimate for coulomb gauge}
Let $n\in \mathbb{Z}$ and $K\in\mathbb{Z}_{>0}$, and let $u\in \Omega^0(\ad E)$.
Then 
\[ \int_{nT<t<(n+1)T} |d_Au|^2d\vol
    \lesssim 2^{-K}\left(\norm{d_Au}_{n-K}^{n+K}\right)^2 + 
    \norm{\Delta_A u}_{n-K}^{n+K}\norm{u}_{n-K}^{n+K}.\]
\end{lemma}
\begin{proof}
For simplicity, we suppose $n=0$.
Take any $m\in \mathbb{Z}$.
Let $\varphi:\mathbb{R}\to [0,1]$ be a cut-off such that 
$\supp(\varphi)\subset [(m-1)T,(m+2)T]$, $\varphi=1$ on $[mT,(m+1)T]$ and 
$|\varphi'|, |\varphi''| < 10/T$.
Then 
\[ \int_{mT<t< (m+1)T}|d_Au|^2 \leq \int_X |d_A(\varphi u)|^2 
    = \int_X \langle \Delta_A(\varphi u),\varphi u\rangle .\]
We have $\Delta_A(\varphi u) = \varphi\Delta_A u + \Delta \varphi \cdot u +*(*d\varphi\wedge d_Au
-d\varphi\wedge *d_Au)$.
\[ |\Delta_A(\varphi u)|\leq (10/T)|u| + (20/T)|d_A u| + |\Delta_A u|.\]
Since $\Delta_A(\varphi u) = \Delta_A u$ over $mT\leq t\leq (m+1)T$,
\begin{equation*}
 \begin{split}
  \int_{mT < t < (m+1)T}|d_A u|^2 \leq &\int_{\{(m-1)T<t<mT \text{ or } (m+1)T <t<(m+2)T\}} (10/T)|u|^2
  + (20/T)|d_A u||u|  \\
  &+ \int_{(m-1)T<t<(m+2)T} |\Delta_A u| |u|.
 \end{split}
\end{equation*}
Using Lemma \ref{lemma: basic estimates from non-degeneracy} (i), 
the right-hand-side is bounded by 
\[ \frac{10C_1+20\sqrt{C_1}}{T}\int_{\{(m-1)T<t<mT \text{ or } (m+1)T <t<(m+2)T\}} |d_A u|^2
    +  \int_{(m-1)T<t<(m+2)T} |\Delta_A u| |u| .\]
From (\ref{eq: definition of T}), this is bounded by 
\[ \frac{1}{4}\int_{\{(m-1)T<t<mT \text{ or } (m+1)T <t<(m+2)T\}} |d_A u|^2d\vol
    + 3\norm{\Delta_A u}_{m-1}^{m+1} \norm{u}_{m-1}^{m+1}.\]
We define a sequence $a_m$ $(-K\leq m \leq K)$ by 
\[ a_m := \int_{mT<t<(m+1)T}|d_A u|^2d\vol.\]
Then the above implies 
\[ a_m \leq \frac{a_{m-1}+a_{m+1}}{4} + 3\norm{\Delta_A u}_{-K}^{K} \norm{u}_{-K}^{K}\quad 
   (-K+1\leq m \leq K-1).\]
By applying Sublemma \ref{sublemma: elementary estimate on sequence} below to this relation, we get 
\[ a_0 \leq \frac{\max(a_K, a_{-K})}{2^{K-1}} + 18\norm{\Delta_A u}_{-K}^{K} \norm{u}_{-K}^{K}
   \leq \frac{1}{2^{K-1}}\left(\norm{d_Au}_{-K}^{K}\right)^2
    + 18\norm{\Delta_A u}_{-K}^{K} \norm{u}_{-K}^{K}.\]

\begin{sublemma} \label{sublemma: elementary estimate on sequence}
Let $K$ be a positive integer, and let $b\geq 0$ be a real number.
Let $\{a_m\}_{-K\leq m\leq K}$ be a sequence of non-negative real numbers satisfying 
\[ a_m \leq \frac{a_{m-1}+a_{m+1}}{4} + b   \quad (-K+1\leq m\leq K-1). \]
Then we have 
\[ a_0\leq \frac{\max(a_K,a_{-K})}{2^{K-1}} + 6b.\]
\end{sublemma}
\begin{proof}
Set $b_m := \max(a_{-m}, a_m)$ $(0\leq m\leq K)$.
We have $b_0\leq b_1/2+b$.
For $m\geq 1$, we have $b_m\leq (b_{m-1}+b_{m+1})/4+b$, i.e. 
$4b_m\leq b_{m-1}+b_{m+1}+4b$.
Hence, for $m\geq 1$, 
\[ 2(b_m-b_{m-1})\leq 2b_m-b_{m-1} \leq -2b_m+b_{m+1}+4b 
    \leq b_{m+1}-b_m +4b .\]
Thus $b_{m}-b_{m-1}\leq (b_{m+1}-b_m)/2 + 2b$ $(m\geq 1)$.
Using this inequality recursively, we get 
\[ b_1-b_0\leq \frac{b_K-b_{K-1}}{2^{K-1}} + 2b\left(1+\frac{1}{2}+\dots+\frac{1}{2^{K-2}}\right)
    \leq  \frac{b_K-b_{K-1}}{2^{K-1}} + 4b  .\]
On the other hand, $2b_0-b_1\leq 2b$.
Hence
\[ a_0= b_0\leq \frac{b_K-b_{K-1}}{2^{K-1}}+6b \leq \frac{b_K}{2^{K-1}} + 6b 
    =\frac{\max(a_K,a_{-K})}{2^{K-1}} + 6b .\]
\end{proof}
We have finished the proof of Lemma \ref{lemma: preliminary technical estimate for coulomb gauge}.
\end{proof}

The next proposition is the main result of this section.
Recall that we have supposed that $A$ is a non-degenerate ASD connection on $E$
with $\norm{F_A}_{\mathrm{op}} \leq d$.
\begin{proposition} \label{prop: Coulomb gauge}
For any $\tau>0$, there exist $\varepsilon_2=\varepsilon_2(A,\tau)>0$ and
$K=K(A,\tau)\in \mathbb{Z}_{>0}$ satisfying the following statement.

Let $n\in \mathbb{Z}$.
Let $a,b\in \Omega^1(\ad E)$ with $d_A^* a=d_A^* b=0$, and let $g:E\to E$ be a gauge transformation.
Set $\alpha := g(A+a)-(A+b)$.
If the $L^\infty$-norms of $a$, $b$ and $\alpha$ 
over $(n-K)T<t<(n+K+1)T$ are all less than  
$\varepsilon_2$, then 
\begin{equation} \label{eq: statement 1 of prop: Coulomb gauge}
 \norm{a-b}_{L^2(nT<t<(n+1)T)}\leq \tau \norm{a-b}_{n-K}^{n+K}
    + \sqrt{\norm{\alpha}_{n-K}^{n+K} + \norm{d_A^* \alpha}_{n-K}^{n+K}},
\end{equation}
\begin{equation} \label{eq: statement 2 of prop: Coulomb gauge}
  \min_{\pm}\norm{g\pm 1}_{L^2(nT<t<(n+1)T)} \lesssim_A
   \norm{\alpha}_{L^2(nT<t<(n+1)T)} + \norm{a-b}_{L^2(nT<t<(n+1)T)}.
\end{equation}
\end{proposition}

\begin{proof}
For simplicity of the notations, we assume $n=0$.
Set $U:=S^3\times (-KT, KT+T)$.
We have $d_A g= -\alpha g+ga-bg$.
Then $|d_Ag|< 3\varepsilon_2$ over $U$.
We choose $\varepsilon_2$ so that $3\varepsilon_2\leq \varepsilon_1$ (the constant introduced in 
Lemma \ref{lemma: estimate on gauge transformation from non-degeneracy}).
Then by Lemma \ref{lemma: estimate on gauge transformation from non-degeneracy}, 
we can suppose
$\norm{g-1}_{L^\infty(U)}\lesssim_A \varepsilon_2\ll 1$.
So there is a section $u$ of $\Lambda^0(\ad E)$ over $U$
satisfying $g=e^u$ and $\norm{u}_{L^\infty(U)}\lesssim_A \varepsilon_2\ll 1$.
Then $2^{-1}|d_Au|\leq |d_Ag|\leq 2|d_Au|$ and $|g-1|\leq 2|u|$ over $U$.
By Lemma \ref{lemma: basic estimates from non-degeneracy} (i),
\begin{equation} \label{eq: |g-1| leq |d_A g|}
 \norm{g-1}_{-K}^{K}\leq 2\norm{u}_{-K}^K \leq 2\sqrt{C_1} \norm{d_A u}_{-K}^K 
   \leq 4\sqrt{C_1}\norm{d_A g}_{-K}^K.          
\end{equation}
We have 
\begin{equation}  \label{eq: d_A g = -alpha g + (g-1)a-b(g-1)+(a-b)}
   d_A g = -\alpha g + (g-1)a-b(g-1)+(a-b).
\end{equation}
Then 
\begin{equation*}
 \begin{split}
 \norm{d_Ag}_{-K}^K &\leq \norm{\alpha}_{-K}^K + \norm{g-1}_{-K}^K
 \left(\norm{a}_{L^\infty(U)} +\norm{b}_{L^\infty(U)}\right)
    +\norm{a-b}_{-K}^K \\
    &\leq  \norm{\alpha}_{-K}^K + \norm{a-b}_{-K}^K  + 8\varepsilon_2\sqrt{C_1} \norm{d_Ag}_{-K}^K.
 \end{split}
\end{equation*}
We choose $\varepsilon_2>0$ so small that $8\varepsilon_2\sqrt{C_1}\leq 1/2$.
Then 
\begin{equation} \label{eq: |d_Ag| leq |alpha|+|a-b|}
 \norm{d_A g}_{-K}^K\leq 2\left(\norm{\alpha}_{-K}^K + \norm{a-b}_{-K}^K\right).
\end{equation}
This and (\ref{eq: |g-1| leq |d_A g|}) shows
\begin{equation} \label{eq: revisit of statement 2 of prop: coulomb gauge}
 \norm{g-1}_{-K}^K \leq 8\sqrt{C_1}\left(\norm{\alpha}_{-K}^K + \norm{a-b}_{-K}^K\right).
\end{equation}
In the same way we get (\ref{eq: statement 2 of prop: Coulomb gauge}):
\[ \norm{g-1}_{L^2(0<t<T)} \leq 8\sqrt{C_1}(\norm{\alpha}_{L^2(0<t<T)}+\norm{a-b}_{L^2(0<t<T)}).\]

From (\ref{eq: d_A g = -alpha g + (g-1)a-b(g-1)+(a-b)}),
\begin{equation*}
 \begin{split}
 \norm{a-b}_{L^2(0<t<T)} &\leq \norm{d_Ag}_{L^2(0<t<T)} +\norm{\alpha}_{L^2(0<t<T)}
 + \norm{g-1}_{L^2(0<t<T)} \left(\norm{a}_{L^\infty(U)}+\norm{b}_{L^\infty(U)}\right)\\
 &\leq \norm{d_Ag}_{L^2(0<t<T)} +\norm{\alpha}_{L^2(0<t<T)} +2\varepsilon_2 \norm{g-1}_{L^2(0<t<T)}\\
 &\leq \norm{d_Ag}_{L^2(0<t<T)} + (1+16\varepsilon_2\sqrt{C_1})\norm{\alpha}_{L^2(0<t<T)}
  + 16\varepsilon_2\sqrt{C_1}\norm{a-b}_{L^2(0<t<T)}.
 \end{split}
\end{equation*}
Since $|d_Ag|\leq 2|d_Au|$ and $\varepsilon_2\ll 1$,
\begin{equation} \label{eq: |a-b| leq |d_Au|+|alpha|+varepsilon_2|a-b|}
  \norm{a-b}_{L^2(0<t<T)} \lesssim_A  \norm{d_Au}_{L^2(0<t<T)} 
  +  \norm{\alpha}_{L^2(0<t<T)}.
\end{equation}

We have the Coulomb gauge condition $d_A^*a=d_A^*b=0$. Therefore 
$\Delta_A g= -*d_A*d_Ag=-*d_A(-*\alpha g+g*a-*b g) =
  -(d_A^*\alpha)g -*(*\alpha\wedge d_Ag)  -*(d_Ag\wedge *a)-*(*b\wedge d_Ag)$.
By (\ref{eq: |d_Ag| leq |alpha|+|a-b|}),
\begin{equation} \label{eq: |Delta_A g| leq |d^*_A alpha| + varepsilon_2(|alpha|+|a-b|)}
 \begin{split}
  \norm{\Delta_A g}_{-K}^K &\leq \norm{d_A^*\alpha}_{-K}^K + 
  \left(\norm{\alpha}_{L^\infty(U)}+\norm{a}_{L^\infty(U)}+\norm{b}_{L^\infty(U)}\right)\norm{d_Ag}_{-K}^K\\
  &\leq  \norm{d_A^*\alpha}_{-K}^K + 6\varepsilon_2\left(\norm{\alpha}_{-K}^K + \norm{a-b}_{-K}^K\right).
 \end{split}
\end{equation}
$\Delta_A g= \sum_{n=0}^\infty\Delta_A(u^n/n!)$ and 
$|\Delta_A(u^n)|\leq n(n-1)|u|^{n-2}|d_Au|^2 + n|u|^{n-1}|\Delta_A u|$.
Hence 
\[ |\Delta_Ag-\Delta_Au|\leq e^{|u|}|d_Au|^2+(e^{|u|}-1)|\Delta_Au| 
    \lesssim \varepsilon_2 \left(|d_Ag|+|\Delta_Au|\right) \]
over $U$.
Here we have used $|u|\lesssim_A \varepsilon_2\ll 1$ and $|d_Au|\leq 2|d_Ag| < 6\varepsilon_2$ over $U$.
We choose $\varepsilon_2$ so small that $|\Delta_A u|\lesssim |\Delta_A g| + \varepsilon_2 |d_A g|$ over $U$.
By (\ref{eq: |d_Ag| leq |alpha|+|a-b|}) and (\ref{eq: |Delta_A g| leq |d^*_A alpha| + varepsilon_2(|alpha|+|a-b|)}),
\begin{equation} \label{eq: |Delta_A u| leq |alpha|+|d_A alpha|+varepsilon_2 |a-b|}
 \norm{\Delta_A u}_{-K}^K \lesssim \varepsilon_2 \norm{\alpha}_{-K}^K+\norm{d_A^*\alpha}_{-K}^K
    + \varepsilon_2\norm{a-b}_{-K}^K.
\end{equation}

From (\ref{eq: |a-b| leq |d_Au|+|alpha|+varepsilon_2|a-b|}),
Lemma \ref{lemma: preliminary technical estimate for coulomb gauge} and 
$\norm{\alpha}_{L^\infty(U)} < \varepsilon_2$,
\begin{equation*}
 \begin{split}
   \norm{a-b}_{L^2(0<t<T)}^2 & \lesssim_A \left(\norm{d_Au}_{L^2(0<t<T)}\right)^2
   + \left(\norm{\alpha}_{L^2(0<t<T)}\right)^2  \\
   & \lesssim_A 2^{-K}\left(\norm{d_Au}_{-K}^K\right)^2 + 
   \norm{\Delta_Au}_{-K}^K\norm{u}_{-K}^K 
   + \varepsilon_2 \norm{\alpha}_{-K}^K.
 \end{split}
\end{equation*}
From (\ref{eq: |d_Ag| leq |alpha|+|a-b|}), $|d_Au|\leq 2|d_Ag|$ on $U$ and 
$\norm{\alpha}_{L^\infty(U)}\leq \varepsilon_2$,
\[ \left(\norm{d_Au}_{-K}^K\right)^2 \lesssim \left(\norm{\alpha}_{-K}^K\right)^2
   + \left(\norm{a-b}_{-K}^K\right)^2
   \lesssim_A  \varepsilon_2\norm{\alpha}_{-K}^K
    + \left(\norm{a-b}_{-K}^K\right)^2 .\]
From (\ref{eq: revisit of statement 2 of prop: coulomb gauge}),
$\norm{u}_{-K}^K \lesssim_A \norm{\alpha}_{-K}^K+\norm{a-b}_{-K}^K$.
From (\ref{eq: |Delta_A u| leq |alpha|+|d_A alpha|+varepsilon_2 |a-b|}) and
$\norm{a}_{L^\infty(U)}, \norm{b}_{L^\infty(U)}, \norm{\alpha}_{L^\infty(U)}< \varepsilon_2$,
\begin{equation*}
 \begin{split}
  \norm{\Delta_A u}_{-K}^K\norm{u}_{-K}^K \lesssim_A &
  \left(\varepsilon_2\norm{\alpha}_{-K}^K 
  +\norm{d_A^*\alpha}_{-K}^K + \varepsilon_2\norm{a-b}_{-K}^K \right)\left(\norm{\alpha}_{-K}^K+\norm{a-b}_{-K}^K\right) \\
   \lesssim_A &\, \varepsilon_2\left(\norm{\alpha}_{-K}^K+\norm{d_A^*\alpha}_{-K}^K\right) 
     + \varepsilon_2\left(\norm{a-b}_{-K}^K\right)^2.
 \end{split}
\end{equation*}
(The strange square root in (\ref{eq: statement 1 of prop: Coulomb gauge}) comes from the term 
$\norm{d_A^*\alpha}_{-K}^K\norm{a-b}_{-K}^K$ in this estimate.)
Thus 
\[ \norm{a-b}_{L^2(0<t<T)}^2\lesssim_A 
    (\varepsilon_2+2^{-K})\left(\norm{a-b}_{-K}^K\right)^2
   + \varepsilon_2\left(\norm{\alpha}_{-K}^K+\norm{d_A^*\alpha}_{-K}^K\right).\]
We choose $K>0$ sufficiently large and $\varepsilon_2>0$ sufficiently small.
Then we get
\[  \norm{a-b}_{L^2(0<t<T)}^2\leq \tau^2\left(\norm{a-b}_{-K}^K\right)^2
    + \norm{\alpha}_{-K}^K+\norm{d_A^*\alpha}_{-K}^K.\]
\end{proof}

\begin{corollary} \label{cor: gauge fixing}
Suppose that $a,b\in \Omega^1(\ad E)$ satisfy $d_A^*a=d_A^*b=0$ and 
$\norm{a}_{L^\infty(X)}, \norm{b}_{L^\infty(X)}\leq \varepsilon_2(A,1/2)$
(the constant introduced in Proposition \ref{prop: Coulomb gauge} for $\tau =1/2$).
If a gauge transformation $g:E\to E$ satisfies $g(A+a)=A+b$, then 
$a=b$ and $g=\pm 1$.
\end{corollary}
\begin{proof}
For any $n\in \mathbb{Z}$, from Proposition \ref{prop: Coulomb gauge} 
(\ref{eq: statement 1 of prop: Coulomb gauge}),
\[ \norm{a-b}_{L^2(nT<t<(n+1)T)}\leq \frac{1}{2}\norm{a-b}_{n-K}^{n+K}
    \leq \frac{1}{2}\sup_{m\in \mathbb{Z}}\norm{a-b}_{L^2(mT<t<(m+1)T)}.\]
Hence 
\[ \sup_{m\in \mathbb{Z}}\norm{a-b}_{L^2(mT<t<(m+1)T)} \leq
   \frac{1}{2}\sup_{m\in \mathbb{Z}}\norm{a-b}_{L^2(mT<t<(m+1)T)} .\]
This implies $a=b$.
Then Proposition \ref{prop: Coulomb gauge} (\ref{eq: statement 2 of prop: Coulomb gauge})
shows $g=\pm 1$.
\end{proof}

\section{Parameter space of the deformation} \label{section: parameter space of the deformation}

For a connection $A$ on $E$, we set 
$D_A:=d_A^*+d_A^+:\Omega^1(\ad E)\to \Omega^0(\ad E)\oplus \Omega^+(\ad E)$.
Here $d_A^*$ is the formal adjoint of $d_A:\Omega^0(\ad E)\to \Omega^1(\ad E)$, and 
$d_A^+$ is the self-dual part of $d_A:\Omega^1(\ad E)\to \Omega^2(\ad E)$.
We define a linear space $H^1_A$ by 
\begin{equation} \label{eq: definition of H^1_A}
 H^1_A := \{a\in \Omega^1(\ad E)|\, D_A a =0, \, \norm{a}_{L^\infty(X)} < \infty\}.
\end{equation}
$(H^1_A, \norm{\cdot}_{L^\infty(X)})$ is a (possibly infinite dimensional) Banach space.
This space will be the parameter space of the deformation theory developed in the next section.
The main purpose of this section is to prove the following proposition:
\begin{proposition} \label{prop: main result of the study of H^1_A}
Let $A$ be a non-degenerate ASD connection on $E$ satisfying $\norm{F_A}_{\mathrm{op}}\leq d$.
Then for any interval $(\alpha,\beta)\subset \mathbb{R}$ of length $> 2$ 
there exists a finite dimensional linear subspace 
$V\subset H^1_A$ satisfying the following two conditions. 

\noindent 
(i) 
\[ \dim V\geq \frac{1}{\pi^2}\int_{\alpha<t<\beta}|F_A|^2d\vol - C_3(A). \]
Here $C_3(A)$ is a positive constant depending only on $A$.
The important point is that it is independent of the interval $(\alpha,\beta)$.

\noindent 
(ii) All $a\in V$ satisfy $\norm{a}_{L^\infty(X)}\leq 2\norm{a}_{L^\infty(\alpha<t<\beta)}$.
\end{proposition}

The following is a preliminary version of Proposition \ref{prop: main result of the study of H^1_A}:
\begin{proposition} \label{prop: preliminary result on the study of H^1_A}
Let $A$ be an ASD connection on $E$ satisfying $\norm{F_A}_{\mathrm{op}}\leq d$.
For any $\varepsilon>0$ and any interval $(\alpha, \beta) \subset \mathbb{R}$ of length $>2$, there exists
a finite dimensional linear subspace $W\subset \Omega^1(\ad E)$ such that 

\noindent
(i) 
\[ \dim W\geq \frac{1}{\pi^2}\int_{\alpha<t<\beta}|F_A|^2d\vol -C(\varepsilon).\]

\noindent 
(ii) All $a\in W$ satisfy $\supp (a) \subset (\alpha,\beta)\times S^3$.

\noindent 
(iii) All $a\in W$ satisfy $\supp(D_A a)\subset (\alpha,\alpha+1)\times S^3\cup  (\beta-1,\beta)\times S^3$
and $\norm{D_Aa}_{L^\infty(X)}\leq \varepsilon \norm{a}_{L^\infty(X)}$.
\end{proposition}
\begin{proof}
From the compactness of $\moduli_d$, there is a bundle trivialization $g$ of $E$ over 
$U:=\{\alpha-1<t<\alpha+1\}\cup \{\beta-1<t<\beta+1\} \subset X$ such that the connection matrix 
$g(A)$ satisfies 
\[ \norm{g(A)}_{C^k(U)} \lesssim C(k) \quad (\forall k\geq 0).\]
Let $\psi: \mathbb{R}\to [0,1]$ be a cut-off function such that 
$\psi =1$ over a small neighborhood of $[\alpha+1, \beta-1]$, 
$\supp(\psi)\subset (\alpha+1/2,\beta-1/2)$ and $|d\psi|\leq 4$.
Define a connection $A'$ over $(\alpha-1,\beta+1)\times S^3$ by $A' := \psi A$.
(The precise definition is as follows: $A'$ is equal to $A$ on a small neighborhood of 
$[\alpha+1,\beta-1]\times S^3$, and 
it is equal to $g^{-1}(\psi g(A))$ over $U$.)
We have $F(A')=\psi F(A)+d\psi\wedge A +(\psi^2-\psi)A^2$.
\[ |F(A')|\leq d + 4|A| + |A^2| \lesssim 1.\]

Set $X':=(\mathbb{R}/(\beta-\alpha)\mathbb{Z})\times S^3$, and let $\pi:X\to X'$ be the natural projection.
We define a principal $SU(2)$ bundle $E'$ on $X'$ as follows:
We identify the region $\{\alpha<t<\beta\}\subset X$ with its projection $\pi\{\alpha<t<\beta\}$ and set 
\[ E' := E|_{\alpha<t<\beta} \cup (\pi(U)\times SU(2)),\]
where we glue the two terms of the right-hand-side by using the trivialization $g$.
We can naturally identify the connection $A'$ with a connection on $E'$ (also denoted by $A'$).
\[ c_2(E') = \frac{1}{8\pi^2} \int_{X'}\mathrm{tr}(F_{A'}^2)
    \geq \frac{1}{8\pi^2}\int_{\alpha<t<\beta}|F_A|^2d\vol  -\const.\]
Let $H^1_{A'}$ be the linear space of $a\in \Omega^1_{X'}(\ad E')$ satisfying $D_{A'}a = d_{A'}^*a+d_{A'}^+a=0$.
From the Atiyah--Singer index theorem,
\[ \dim H^1_{A'} \geq  8c_2(E') 
    \geq \frac{1}{\pi^2}\int_{\alpha<t<\beta}|F_A|^2d\vol  -\const.\]
\begin{lemma}\label{lemma: in prop: preliminary result on the study of H^1_A}
All $a\in H^1_{A'}$ satisfy
\[ \norm{\nabla_{A'}a}_{L^\infty(X')} \lesssim \norm{a}_{L^\infty(X')}.\]
\end{lemma}    
\begin{proof}
Take any $\gamma\in \mathbb{R}$.
From the construction, we can choose a connection matrix of $A'$ over $\pi\{\gamma<t<\gamma+1\}$ so that 
\[ \norm{A'}_{C^k(\pi\{\gamma<t<\gamma+1\})} \lesssim C(k) \quad (\forall k\geq 0).\]
Then the standard elliptic regularity theory 
(Gilbarg--Trudinger \cite[Theorem 9.11]{Gilbarg-Trudinger}) shows 
\[ \norm{\nabla_{A'} a}_{L^\infty(\pi\{\gamma+1/4<t<\gamma+3/4\})} \lesssim 
   \norm{a}_{L^\infty (\pi\{\gamma<t<\gamma+1\})}.\]
A similar argument will be also used in the proof of Lemma \ref{lemma: L^infty estimate of nabla_B phi}.
\end{proof}
Set $\Omega := \pi(U) \subset X'$.
Let $\tau = \tau(\varepsilon)>0$ be a small number which will be fixed later. 
Take points $x_1, x_2,\dots,x_N$ $(N \lesssim 1/\tau^4)$ in $\Omega$ 
such that for any $x\in \Omega$ there is some $x_i$ satisfying 
$d(x, x_i)\leq \tau$.
Let $V$ be the kernel of the following linear map:
\[ H^1_{A'}\to \bigoplus_{i=1}^N (\Lambda^1(\ad E'))_{x_i}, \quad a\mapsto (a(x_i))_{i=1}^N.\]
We have 
\begin{equation*}
   \dim V\geq \dim H^1_{A'}-12 N\geq \frac{1}{\pi^2}\int_{\alpha<t<\beta}|F_A|^2d\vol -\const -12N.
\end{equation*}

Take any $a\in V$ and $x\in \Omega$.
Choose $x_i$ satisfying $d(x,x_i)\leq \tau$.
From Lemma \ref{lemma: in prop: preliminary result on the study of H^1_A} and $a(x_i)=0$,
\[ |a(x)|\leq \tau \norm{\nabla_{A'} a}_{L^\infty(X')}\lesssim \tau \norm{a}_{L^\infty(X')}.\]
We can choose $\tau >0$ so that
the maximum of $|a|$ is attained at a point in $X'\setminus \Omega$.
For $a\in V$, we define $\tilde{a}\in \Omega^1(\ad E)$ over $X$ by $\tilde{a} :=\psi a$.
(The precise definition is as follows:
We identify the region $\{\alpha<t<\beta\}$ with its projection in $X'$.
$\tilde{a}$ is equal to $\psi a$ over $\alpha<t<\beta$, and it is equal to $0$ outside of $\supp (\psi)$.)
Set $W:=\{\tilde{a}|\, a\in V\} \subset \Omega^1(\ad E)$.
$W$ satisfies the condition (ii) in the statement.
We have $\norm{\tilde{a}}_{L^\infty(X)}=\norm{a}_{L^\infty(X')}$ 
because the maximum of $|a|$ is attained at a point in $X'\setminus \Omega$. Hence
\[ \dim W = \dim V \geq \frac{1}{\pi^2}\int_{\alpha<t<\beta}|F_A|^2d\vol -\const -12N.\]

We have 
\[ D_A \tilde{a} = (A-A')*\tilde{a} + D_{A'}(\psi a) = (A-A')*\tilde{a} + (d\psi)*a.\]
Here $*$ are algebraic multiplications.
$D_{A}\tilde{a}$ is supported in $\{\alpha<t<\alpha+1\}\cup\{\beta-1<t<\beta\}$ and 
\[ \norm{D_A\tilde{a}}_{L^\infty(X)} \lesssim 
  \norm{a}_{L^\infty((\alpha,\alpha+1)\times S^3 \cup (\beta-1,\beta)\times S^3)} \lesssim
   \tau \norm{\tilde{a}}_{L^\infty(X)}.\] 
We can choose $\tau = \tau(\varepsilon)>0$ so that 
$\norm{D_A\tilde{a}}_{L^\infty(X)} \leq \varepsilon \norm{\tilde{a}}_{L^\infty(X)}$. 
Then $W$ satisfies the conditions (i), (ii), (iii) in the statement.
\end{proof}

\begin{lemma} \label{lemma: killing the error term}
Let $\alpha<\beta$.
Let $A$ be an ASD connection on $E$ satisfying $\norm{F_A}_{\mathrm{op}}\leq d$.

\noindent 
(i) If $A$ is non-degenerate, then there is a linear map 
\[ \{u\in \Omega^0(\ad E)|\, \supp(u) \subset (\alpha,\alpha+1)\times S^3 \cup (\beta-1,\beta)\times S^3\}
    \to \Omega^0(\ad E), \quad u\mapsto v, \]
satisfying
\[ d^*_A d_A v=u, \quad 
   \norm{v}_{L^\infty(X)} +\norm{d_Av}_{L^\infty(X)} \lesssim_A \norm{u}_{L^\infty(X)}.\]
   
\noindent 
(ii) There is a linear map 
\[ \{\xi\in \Omega^+(\ad E)|\, \supp(\xi)\subset (\alpha,\alpha+1)\times S^3 \cup (\beta-1,\beta)\times S^3\}
   \to \Omega^+(\ad E), \quad \xi \mapsto \eta,\]
satisfying
\[ d_A^+d_A^*\eta =\xi, \quad \norm{\eta}_{L^\infty(X)}+\norm{\nabla_A\eta}_{L^\infty(X)}
   \lesssim \norm{\xi}_{L^\infty(X)}.\]  
The statement (ii) does not require the non-degeneracy of $A$. 
\end{lemma}
\begin{proof}
(i) Set $L^2_{1,A}(\ad E):= \{w\in L^2(\ad E)|\, d_Aw\in L^2(X)\}$ with the inner product
$(w_1,w_2)' := (d_Aw_1,d_Aw_2)_{L^2(X)}$.
From Lemma \ref{lemma: basic estimates from non-degeneracy} (i), 
every compactly supported $w\in \Omega^0(\ad E)$ satisfies
$\norm{w}_{L^2(X)}\leq \sqrt{C_1}\norm{d_Aw}_{L^2(X)} =\sqrt{C_1}\norm{w}'$.
Hence the norm $\norm{\cdot}'$ is equivalent to $\norm{\cdot}_{L^2_{1,A}(X)}$.
In particular $(L^2_{1,A}(\ad E),(\cdot,\cdot)')$ becomes a Hilbert space.

The rest of the argument is the standard $L^2$-method:
Take $u\in \Omega^0(\ad E)$ with $\supp(u) \subset (\alpha,\alpha+1)\times S^3 \cup (\beta-1,\beta)\times S^3$.
We apply the Riesz representation theorem to the following bounded linear functional:
\[ (\cdot,u)_{L^2(X)}: L^2_{1,A}(\ad E)\to \mathbb{R}, \quad w\mapsto (w,u)_{L^2(X)}.\]
(From Lemma \ref{lemma: basic estimates from non-degeneracy} (i), 
$|(w,u)_{L^2(X)}| \leq \sqrt{C_1}\norm{w}' \norm{u}_{L^2(X)}$.)
Then there uniquely exists $v\in L^2_{1,A}(\ad E)$ satisfying 
$(d_Aw,d_Av)=(w,v)' =(w,u)_{L^2(X)}$.
This means that $d^*_A d_A v=u$ as a distribution.
Moreover $\norm{d_A v}_{L^2(X)} = \norm{v}' \leq \sqrt{C_1}\norm{u}_{L^2(X)} 
\lesssim_A \norm{u}_{L^\infty(X)}$.
From Lemma \ref{lemma: basic estimates from non-degeneracy} (i),
$\norm{v}_{L^2(X)}\lesssim_A \norm{u}_{L^\infty(X)}$.
As in the proof of Lemma \ref{lemma: in prop: preliminary result on the study of H^1_A}, the elliptic regularity 
theory gives
\[ \norm{v}_{L^\infty(X)} +\norm{d_Av}_{L^\infty(X)}\lesssim 
   \norm{v}_{L^2(X)} + \norm{d^*_A d_A v}_{L^\infty(X)} 
   \lesssim_A \norm{u}_{L^\infty(X)}.\]

(ii) We have the Weitzenb\"{o}ck formula \cite[Chapter 6]{Freed-Uhlenbeck}:
$d_A^+d_A^*\eta = \frac{1}{2}\left(\nabla_A^*\nabla_A+S/3\right)\eta$ for $\eta \in \Omega^+(\ad E)$.
Here $S$ is the scalar curvature of $X$, and it is a positive constant.
Then the $L^2$-method shows the above statement.
(Indeed a stronger result will be given in Lemma \ref{lemma: Green kernel estimate}
 in Section \ref{subsection: perturbation}.)
\end{proof}

\begin{proof}[Proof of Proposition \ref{prop: main result of the study of H^1_A}]
Let $\varepsilon = \varepsilon(A)>0$ be a small number which will be fixed later.
For this $\varepsilon$ and the interval $(\alpha,\beta)\subset \mathbb{R}$
there is a finite dimensional subspace $W\subset \Omega^1(\ad E)$ satisfying the conditions (i), (ii), (iii) in 
Proposition \ref{prop: preliminary result on the study of H^1_A}.

From Lemma \ref{lemma: killing the error term},
there is a linear map $W\to \Omega^0(\ad E)\oplus \Omega^+(\ad E)$, $a\mapsto (v, \eta)$, satisfying
$d^*_A d_A v=d_A^*a$, $d_A^+d_A^*\eta = d_A^+a$ and 
\[ \norm{d_Av}_{L^\infty(X)}+\norm{d_A^*\eta}_{L^\infty(X)}\leq 
   C\norm{D_A a}_{L^\infty(X)} \leq \varepsilon C\norm{a}_{L^\infty(X)}
  (= \varepsilon C\norm{a}_{L^\infty(\alpha<t<\beta)}) \]
where $C=C(A)$ is a positive constant depending only on $A$.
Here we have used the conditions (ii) and (iii) in Proposition \ref{prop: preliminary result on the study of H^1_A}.
Set $a':= a-d_Av-d_A^*\eta$. 
This satisfies $D_A a'=0$.
Set $V:= \{a'|\, a\in W\}\subset H^1_A$.
We have $\norm{a'}_{L^\infty(X)}\geq (1-\varepsilon C)\norm{a}_{L^\infty(X)}$ for $a\in W$.
We choose $\varepsilon>0$ sufficiently small so that $(1-\varepsilon C)>0$.
Then $\dim V=\dim W$. From the condition (i) of Proposition \ref{prop: preliminary result on the study of H^1_A},
\[ \dim V\geq \frac{1}{\pi^2}\int_{\alpha<t<\beta}|F_A|^2 d\vol -\const_\varepsilon.\]
We have $\norm{a'}_{L^\infty(X)}\leq (1+\varepsilon C)\norm{a}_{L^\infty(X)}$ for $a\in W$.
On the other hand, from the conditions (ii) and (iii) of Proposition \ref{prop: preliminary result on the study of H^1_A},
\[ \norm{a'}_{L^\infty(\alpha<t<\beta)}\geq 
   \norm{a}_{L^\infty(\alpha<t<\beta)}-\varepsilon C\norm{a}_{L^\infty(X)}
   = (1-\varepsilon C)\norm{a}_{L^\infty(X)}.\]
Hence 
\[ \norm{a'}_{L^\infty(X)}\leq \frac{1+\varepsilon C}{1-\varepsilon C}\norm{a'}_{L^\infty(\alpha<t<\beta)}.\]
We choose $\varepsilon>0$ so that $(1+\varepsilon C)/(1-\varepsilon C)\leq 2$.
Then $\norm{a'}_{L^\infty(X)}\leq 2\norm{a'}_{L^\infty(\alpha<t<\beta)}$ for all $a'\in V$.
\end{proof}

\section{Deformation theory and the proof of Theorem \ref{thm: local mean dimension around non-degenerate ASD connection}} 
\label{section: deformation theory}

In this section we develop a deformation theory of non-degenerate ASD connections
and prove Theorem \ref{thm: local mean dimension around non-degenerate ASD connection}.
(The paper \cite{Matsuo-Tsukamoto} studied a deformation theory of periodic ASD connections.)
Let $A$ be a non-degenerate ASD connection on $E$ satisfying 
$\norm{F_A}_{\mathrm{op}} < d$. 
Note that this is a strict inequality.
We fix this $A$ throughout this section.

\subsection{Deformation theory}  \label{subsection: deformation theory}
Let $H^1_A\subset \Omega^1(\ad E)$ be the Banach space defined by (\ref{eq: definition of H^1_A}).
Let $k\geq 0$ and $0\leq i\leq 4$ be integers.
For $\xi \in L^2_{k,loc}(\Lambda^i(\ad E))$ (a locally $L^2_k$-section of $\Lambda^i(\ad E)$),
we set 
\[ \norm{\xi}_{\ell^\infty L^2_k} := \sum_{j=0}^k \sup_{n\in \mathbb{Z}}
    \norm{\nabla_A^j \xi}_{L^2(n<t< n+1)}.\]
From the elliptic regularity,
we have $\norm{a}_{L^\infty(X)} \lesssim \norm{a}_{\ell^\infty L^2_k}\lesssim \const_k \norm{a}_{L^\infty(X)}$
for $a\in H^1_A$ (cf. the proof of Lemma \ref{lemma: in prop: preliminary result on the study of H^1_A}).

Let $\ell^\infty L^2_k(\Lambda^+(\ad E))$ be the Banach space of $\xi\in L^2_{k,loc}(\Lambda^+(\ad E))$
satisfying $\norm{\xi}_{\ell^\infty L^2_k} < \infty$.
From the Sobolev embedding theorem,
$\norm{\xi}_{L^\infty(X)}\lesssim \norm{\xi}_{\ell^\infty L^2_3}$ for $\xi \in \ell^\infty L^2_3(\Lambda^+(\ad E))$.
Consider 
\begin{equation*}
 \begin{split}
 \Phi: H^1_A &\times \ell^\infty L^2_5(\Lambda^+(\ad E))\to \ell^\infty L^2_3(\Lambda^+(\ad E)), \\
   & (a,\phi)\mapsto F^+(A+a+d_A^*\phi) = (a\wedge a)^+ +d_A^+d_A^*\phi + [a\wedge d_A^*\phi]^+ + (d_A^*\phi \wedge d_A^*\phi)^+.
 \end{split}
\end{equation*}
This is a smooth map between the Banach spaces with $\Phi(0,0)=0$.
We want to describe the fiber $\Phi^{-1}(0)$ around the origin by using the implicit function theorem.
Let $(\partial_2\Phi)_{0}:\ell^\infty L^2_5(\Lambda^+(\ad E))\to \ell^\infty L^2_3(\Lambda^+(\ad E))$
be the derivative of $\Phi$ at the origin with respect to the second variable $\phi$.
We have $(\partial_2\Phi)_{0}(\phi) = d_A^+d_A^*\phi = \frac{1}{2}(\nabla_A^*\nabla_A +S/3)\phi$
for $\phi\in \ell^\infty L^2_5(\Lambda^+(\ad E))$ by the Weitzenb\"{o}ck formula.
($S$ is the scalar curvature of $X$, and it is a positive constant.)
The following $L^\infty$-estimate is proved in \cite[Proposition A.5]{Matsuo-Tsukamoto}:
\begin{lemma} \label{lemma: L^infty estimate}
Let $\xi$ be a $C^2$-section of $\Lambda^+(\ad E)$ over $X$.
We set $\eta:= (\nabla_A^*\nabla_A +S/3)\xi$, and suppose 
$\norm{\xi}_{L^\infty(X)} <\infty$ and $\norm{\eta}_{L^\infty(X)} < \infty$.
Then 
\[ \norm{\xi}_{L^\infty(X)}\leq (24/S)\norm{\eta}_{L^\infty(X)}.\]
\end{lemma}
\begin{lemma} \label{lemma: deformation is unobstructed}
The operator 
$(\partial_2\Phi)_{0}: \ell^\infty L^2_5(\Lambda^+(\ad E))\to \ell^\infty L^2_3(\Lambda^+(\ad E))$
is an isomorphism.
This means that a local deformation of $A$ is ``unobstructed''.
\end{lemma}
\begin{proof}
This can be proved by using Lemma \ref{lemma: Green kernel estimate}
in Section \ref{subsection: perturbation}.
 But here we give a direct proof.
From the $L^\infty$-estimate in Lemma \ref{lemma: L^infty estimate}, 
the above operator is injective.
Hence the problem is its surjectivity.
Take $\eta\in \ell^\infty L^2_3(\Lambda^+(\ad E))$.
Let $\varphi_n:\mathbb{R}\to [0,1]$ be a cut-off function such that 
$\varphi_n =1$ over $[-n,n]$ and $\supp(\varphi_n)\subset (-n-1,n+1)$.
Set $\eta_n:=\varphi_n \eta$.
By the $L^2$-method (see the proof of Lemma \ref{lemma: killing the error term}), 
there exists $\xi_n\in L^2_{1,A}(\Lambda^+(\ad E))$
satisfying $(\nabla_A^*\nabla_A+S/3)\xi_n = \eta_n$ as a distribution and 
$\norm{\xi_n}_{L^2(X)}\lesssim \norm{\eta_n}_{L^2(X)} < \infty$.
From the elliptic regularity, $\xi_n$ is in $L^2_{5,loc}$ and hence of class $C^2$.
Moreover
$\norm{\xi_n}_{L^\infty(X)}\lesssim \norm{\xi_n}_{L^2(X)}+\norm{\eta_n}_{L^\infty} < \infty$.
Hence by the $L^\infty$-estimate (Lemma \ref{lemma: L^infty estimate})
\[  \norm{\xi_n}_{L^\infty(X)}\leq (24/S)\norm{\eta_n}_{L^\infty(X)}\leq (24/S)\norm{\eta}_{L^\infty(X)}
     \lesssim \norm{\eta}_{\ell^\infty L^2_3}.\]
For any integer $m$,
\[ \norm{\xi_n}_{L^2_5(m<t<m+1)} \lesssim
    \norm{\xi_n}_{L^\infty(X)}+\norm{\eta_n}_{\ell^\infty L^2_3}
    \lesssim \norm{\eta}_{\ell^\infty L^2_3}.\]
By choosing a subsequence $\{\xi_{n_k}\}_{k\geq 1}$, there exists $\xi\in L^2_{5,loc}(\Lambda^+(\ad E))$ such that 
$\xi_{n_k}$ converges to $\xi$ weakly in $L^2_5((m, m+1)\times S^3)$ for every $m\in \mathbb{Z}$.
Then $(\nabla_A^*\nabla_A+S/3)\xi=\eta$ and 
$\norm{\xi}_{\ell^\infty L^2_5}\lesssim \norm{\eta}_{\ell^\infty L^2_3} < \infty$.
\end{proof}
By the implicit function theorem, we can choose $R>0$ and $R'>0$ such that for any 
$a\in H^1_A$ with $\norm{a}_{L^\infty(X)}\leq R$ there uniquely exists
$\phi_a\in \ell^\infty L^2_5(\Lambda^+(\ad E))$
satisfying $F^+(A+a+d_A^*\phi_a)=0$ and $\norm{\phi_a}_{\ell^\infty L^2_5}\leq R'$.
We have $\phi_0=0$.
For $a\in B_R(H^1_A) := \{a\in H^1_A| \norm{a}_{L^\infty}\leq R\}$ we set $a':=a+d_A^*\phi_a$.
This satisfies the ASD equation $F^+(A+a')=0$ and the Coulomb gauge condition $d_A^*a'=0$.
Since $\norm{F_A}_{\mathrm{op}}<d$, we can choose $R>0$ sufficiently small so that 
$\norm{F(A+a')}_{\mathrm{op}}\leq d$ for all $a\in B_{R}(H^1_A)$.
Thus we get a deformation map:
\begin{equation}  \label{eq: deformation map}
  B_R(H^1_A)\to \moduli_d, \quad a\mapsto [A+a'].
\end{equation}

The derivative $(\partial_1\Phi)_{0}:H^1_A\to \ell^\infty L^2_3(\Lambda^+(\ad E))$ of $\Phi$
at the origin with respect to the first variable is equal to zero.
Hence the derivative of the following map at the origin is also zero:
\[ B_R(H^1_A)\to \ell^\infty L^2_5(\Lambda^+(\ad E)), \quad a\mapsto \phi_a. \]
Then we get 
\begin{equation}  \label{eq: phi_a - phi_b}
   \norm{\phi_a-\phi_b}_{\ell^\infty L^2_5} \lesssim_A 
   \left(\norm{a}_{L^\infty(X)}+ \norm{b}_{L^\infty(X)}\right) \norm{a-b}_{L^\infty(X)}
\end{equation}
for $a,b\in B_R(H^1_A)$.
In particular the map $(B_R(H^1_A), \norm{\cdot}_{L^\infty(X)})\to \moduli_d$ is continuous.

\begin{remark}
Note that the construction of the deformation map (\ref{eq: deformation map}) 
does not use the non-degeneracy condition of $A$.
It will be used for the further study of the deformation map.
Indeed, since $A$ is non-degenerate, we can apply Corollary \ref{cor: gauge fixing} to this situation.
Then we can show that the above map (\ref{eq: deformation map}) is injective if $R$ is sufficiently small.
Moreover if $B_R(H^1_A)$ is endowed with the topology of uniform convergence over compact subsets (this is not 
equal to the norm topology), then $B_R(H^1_A)$ is compact 
and the map (\ref{eq: deformation map}) becomes a topological embedding.
We don't need these facts for the proof of Theorem \ref{thm: local mean dimension around non-degenerate ASD connection}.
So we omit the detail. But it is not difficult.
\end{remark}

\begin{remark}
In the above argument we have solved the equation $F^+(A+a+d_A^*\phi)=0$ by using the implicit function theorem.
But indeed we can solve it more directly by using the method of Section \ref{subsection: perturbation}.
So there exists a little redundancy in our way of the explanation.
We can prepare a unified method for both Sections \ref{subsection: deformation theory} and 
\ref{subsection: perturbation}.
But we don't take this way here because 
this redundancy is not so heavy and the above implicit function theorem argument seems conceptually easier 
(at least for the authors) to understand.
\end{remark}

\subsection{Proof of Theorem \ref{thm: local mean dimension around non-degenerate ASD connection}}
\label{subsection: proof of Theorem local mean dimension around A}

We need a distance on $\moduli_d$.
Any choice will do.
One choice is:
For $[A_1],[A_2]\in \moduli_d$, we define the distance $\dist([A_1],[A_2])$ as the infimum of 
\[ \sum_{n=1}^\infty 2^{-n} \frac{\norm{g(A_1)-A_2}_{L^\infty(-n<t<n)}}{1+\norm{g(A_1)-A_2}_{L^\infty(-n<t<n)}} \]
over all gauge transformations $g:E\to E$.
We don't need this explicit formula. But probably it will be helpful for the understanding.

Recall the following notation:
For $\Omega\subset \mathbb{R}$ we define $\dist_{\Omega}([A_1],[A_2])$ as the supremum of
$\dist([s^*(A_1)],[s^*(A_2)])$ over $s\in \Omega$.
$s^*(\cdot)$ is the pull-back by $s:E\to E$.
In particular, for $s\in \mathbb{R}$, the distance $\dist_{\{s\}}([A_1],[A_2])$ is the infimum of 
\[  \sum_{n=1}^\infty 2^{-n} 
   \frac{\norm{g(A_1)-A_2}_{L^\infty(s-n<t<s+n)}}{1+\norm{g(A_1)-A_2}_{L^\infty(s-n<t<s+n)}}\]  
over all gauge transformations $g:E\to E$.
We will abbreviate $\dist_{\{s\}}([A_1],[A_2])$ to $\dist_s([A_1],[A_2])$. 

For the proof of Theorem \ref{thm: local mean dimension around non-degenerate ASD connection},
we need to compare the distances $\dist_{(\alpha,\beta)}$ on $\moduli_d$ and 
$\norm{\cdot}_{L^\infty(\alpha<t<\beta)}$ on $B_R(H^1_A)$ for intervals $(\alpha,\beta)\subset \mathbb{R}$.
The next lemma gives us a solution.
It is a consequence of Proposition \ref{prop: Coulomb gauge}.
\begin{lemma} \label{lemma: distortion of the deformation map}
We can choose $0<R_1 <R$ so that the following 
statement holds.
For any $\delta>0$ there exists $\varepsilon>0$ such that 
if $a,b\in H^1_A$ with $\norm{a}_{L^\infty(X)}, \norm{b}_{L^\infty(X)} < R_1$ satisfy
\[ \dist_{s}([A+a'],[A+b']) \leq \varepsilon \]
for some $s\in \mathbb{R}$, then 
\[ \norm{a-b}_{L^\infty(s<t<s+1)} \leq \frac{1}{4}\norm{a-b}_{L^\infty(X)} + \delta.\]
\end{lemma}
\begin{proof}
Let $T=T(A)>3$ be the positive constant introduced in Section \ref{section: Coulomb gauge}.
See the discussion around (\ref{eq: definition of T}).
We choose $n\in \mathbb{Z}$ so that 
\[ nT \leq s-1<s+2 \leq (n+2)T.\]
Then from the elliptic regularity 
\[ \norm{a-b}_{L^\infty(s<t<s+1)} \lesssim \norm{a-b}_{L^2(s-1<t<s+2)}\leq  \norm{a-b}_{L^2(nT<t<(n+2)T)}.\]

Let $0<\tau<1$ be a small number which will be fixed later.
Let $\varepsilon_2 = \varepsilon_2(A,\tau)>0$ and $K=K(A, \tau)\in \mathbb{Z}_{>0}$ be the positive constants 
introduced in Proposition \ref{prop: Coulomb gauge}.
From (\ref{eq: phi_a - phi_b}),
if $R_1\ll 1$,
\[ \norm{a-b}_{L^2(nT<t<(n+2)T)} \leq \norm{a'-b'}_{L^2(nT<t<(n+2)T)} + \tau \norm{a-b}_{L^\infty(X)}.\]
Hence 
\begin{equation} \label{eq: a-b and a'-b'}
  \norm{a-b}_{L^\infty(s<t<s+1)} \lesssim \norm{a'-b'}_{L^2(nT<t<(n+2)T)} + \tau \norm{a-b}_{L^\infty(X)}.
\end{equation}
We estimate the term $ \norm{a'-b'}_{L^2(nT<t<(n+2)T)}$ by using Proposition \ref{prop: Coulomb gauge}.

We can assume $\delta^2< \varepsilon_2$.
From the Uhlenbeck compactness we can choose $\varepsilon>0$ so that 
if two connections $[A_1],[A_2]\in \moduli_d$ satisfy 
$\dist ([A_1],[A_2]) \leq \varepsilon$ then there exists a gauge transformation $g:E\to E$ satisfying 
\[ \norm{g(A_1)-A_2}_{L^\infty(-KT-2T<t<KT+2T)} + \norm{d_{A_2}^*(g(A_1)-A_2)}_{L^\infty(-KT-2T<t<KT+2T)}
   < \tau^2\delta^2.\]
Then the assumption $\dist_s([A+a'],[A+b']) \leq \varepsilon$ implies that there exists a gauge transformation 
$g:E\to E$ satisfying (set $\alpha := g(A+a')-(A+b')$)
\[  \norm{\alpha}_{L^\infty((n-K)T<t<(n+K+2)T)} + \norm{d_{A+b'}^*\alpha}_{L^\infty((n-K)T<t<(n+K+2)T)}
    < \tau^2 \delta^2 < \varepsilon_2.\]
In particular $\norm{\alpha}_{L^\infty((n-K)T<t<(n+K+2)T)} < \varepsilon_2$.
Hence if $R_1\ll \varepsilon_2$ then we can apply Proposition \ref{prop: Coulomb gauge} to the present situation:
\begin{equation*}
  \begin{split}
   \norm{a'-b'}_{L^2(nT<t<(n+2)T)} &\lesssim_A \tau \norm{a'-b'}_{\ell^\infty L^2}+ \tau\delta \\
   &\lesssim_A \tau \norm{a-b}_{L^\infty(X)} + \tau \delta \quad (\text{by (\ref{eq: phi_a - phi_b})}).
  \end{split}
\end{equation*}
By applying this estimate to the above (\ref{eq: a-b and a'-b'}), we get 
\[ \norm{a-b}_{L^\infty(s<t<s+1)} \lesssim_A \tau \norm{a-b}_{L^\infty(X)} + \tau \delta.\]
We choose $\tau>0$ sufficiently small. Then 
\[ \norm{a-b}_{L^\infty(s<t<s+1)} \leq \frac{1}{4}\norm{a-b}_{L^\infty(X)} + \delta.\]
\end{proof}
Recall that $B_r([A])_{\mathbb{R}}\subset \moduli_d$ is the closed ball of radius $r$ centered at $[A]$
with respect to the distance $\dist_{\mathbb{R}}$.
\begin{proposition} \label{prop: widim estimate for B_r([A])_R}
For any $r>0$ there exists $\varepsilon(r)>0$ such that for any $0<\varepsilon \leq \varepsilon(r)$ and 
any interval $(\alpha,\beta)\subset \mathbb{R}$ of 
length $>2$ we have 
\[ \widim_{\varepsilon} (B_r([A])_{\mathbb{R}},\dist_{(\alpha,\beta)}) \geq 
   \frac{1}{\pi^2}\int_{\alpha<t<\beta}|F_A|^2d\vol -C_3.\]
Here $C_3=C_3(A)$ is the positive constant introduced in Proposition \ref{prop: main result of the study of H^1_A}.
\end{proposition}
\begin{proof}
We can choose $0<r'<R_1$ such that every $a\in B_{r'}(H^1_A)$ satisfies 
$[A+a']\in B_r([A])_{\mathbb{R}}$.
($R_1$ is the constant introduced in the previous lemma.)
From Lemma \ref{lemma: distortion of the deformation map} we can choose $\varepsilon(r)>0$
so that if $a,b\in B_{r'}(H^1_A)$ satisfy 
\[ \dist_{(\alpha,\beta)}([A+a'],[A+b']) \leq \varepsilon(r) \]
then 
\begin{equation} \label{eq: choice of varepsilon(r)}
  \norm{a-b}_{L^\infty(\alpha<t<\beta)} \leq \frac{1}{4}\norm{a-b}_{L^\infty(X)} + \frac{r'}{8}.
\end{equation}

By Proposition \ref{prop: main result of the study of H^1_A}, there exists a linear subspace $V\subset H^1_A$ 
such that 
\[ \dim V\geq  \frac{1}{\pi^2}\int_{\alpha<t<\beta}|F_A|^2d\vol -C_3\]
and that all $a\in V$ satisfy $\norm{a}_{L^\infty(X)} \leq 2\norm{a}_{L^\infty(\alpha<t<\beta)}$.
We investigate the restriction of the deformation map (\ref{eq: deformation map}) to
$B_{r'}(V) := \{a\in V|\norm{a}_{L^\infty(X)} \leq r'\}$.

By applying the above (\ref{eq: choice of varepsilon(r)}) to $B_{r'}(V)$, 
we get the following:
If $a,b\in B_{r'}(V)$ satisfy $\dist_{(\alpha,\beta)}([A+a'],[A+b']) \leq \varepsilon(r)$, then
\[ \norm{a-b}_{L^\infty(X)} \leq 2\norm{a-b}_{L^\infty(\alpha<t<\beta)} \leq 
   \frac{1}{2}\norm{a-b}_{L^\infty(X)} + \frac{r'}{4} \]
and hence $\norm{a-b}_{L^\infty(X)} \leq r'/2$.
Therefore we get: For $0<\varepsilon\leq \varepsilon(r)$
\begin{equation*}
  \begin{split}
   \widim_\varepsilon(B_r([A])_{\mathbb{R}},\dist_{(\alpha,\beta)}) &\geq 
   \widim_{r'/2}(B_{r'}(V),\norm{\cdot}_{L^\infty(X)}) = \dim V 
   \quad (\text{by Example \ref{example: widim of the Banach ball}}) \\
   &\geq \frac{1}{\pi^2}\int_{\alpha<t<\beta}|F_A|^2d\vol -C_3.
  \end{split}
\end{equation*}
\end{proof}

\begin{proof}[Proof of Theorem \ref{thm: local mean dimension around non-degenerate ASD connection}]
The upper bound $\dim_{[A]}(\moduli_d:\mathbb{R})\leq 8\rho(A)$ is given by Theorem \ref{thm: upper bound}.
So the problem is the lower bound.

$\dim_{[A]}(\moduli_d:\mathbb{R}) = \lim_{r\to 0}\dim(B_r([A])_{\mathbb{R}}:\mathbb{R})$, and
$\dim(B_r([A])_{\mathbb{R}}:\mathbb{R})$ is given by 
\[ \lim_{\varepsilon\to 0}\left(\lim_{n\to +\infty}
 \frac{\sup_{x\in \mathbb{R}}\widim_\varepsilon(B_r([A])_{\mathbb{R}},\dist_{(x,x+n)})}{n}\right).\]
By Proposition \ref{prop: widim estimate for B_r([A])_R},
for $0<\varepsilon\leq \varepsilon(r)$ and $n>2$ 
\[ \sup_{x\in \mathbb{R}}\widim_\varepsilon(B_r([A])_{\mathbb{R}},\dist_{(x,x+n)})\geq 
  \frac{1}{\pi^2}\sup_{x\in \mathbb{R}}\int_{x<t<x+n}|F_A|^2d\vol -C_3.\]
Since 
\[ \rho(A) = \lim_{n\to \infty} \frac{1}{8\pi^2n}\sup_{x\in \mathbb{R}}\int_{x<t<x+n}|F_A|^2d\vol,\]
we have 
\[ \dim(B_r([A])_{\mathbb{R}}:\mathbb{R})\geq 8\rho(A).\]
Thus $\dim_{[A]}(\moduli_d:\mathbb{R})\geq 8\rho(A)$.
\end{proof}

\section{Gluing infinitely many instantons} \label{section: gluing infinitely many instantons}

In this section we prove Theorem \ref{thm: abundance of non-degenerate ASD connections}:
Suppose $d>1$. Let $\varepsilon>0$, and let $A$ be an ASD connection on $E$ with $\norm{F_A}_{\mathrm{op}}<d$.
We want to find a non-degenerate ASD connection $\tilde{A}$ on $E$ satisfying
\begin{equation}  \label{eq: condition for tilde A}
   \norm{F(\tilde{A})}_{\mathrm{op}}<d,\quad \rho(\tilde{A})>\rho(A)-\varepsilon.
\end{equation}
If $A$ itself is non-degenerate, then $\tilde{A} := A$ satisfies the condition.
So we assume that $A$ is degenerate.

As we described in Section \ref{subsection: ideas of the proofs},
the idea of the proof is gluing instantons.
We glue infinitely many copies of the instanton $I$ (given in Example \ref{example: BPST instanton}) to $A$ 
over the regions where the curvature $F_A$ has very small norm.
Then we get a non-degenerate ASD connection $\tilde{A}$.
The technique of gluing infinitely many instantons in the context of Yang--Mills theory was first developed 
in \cite{Tsukamoto}. It was further expanded in \cite{Tsukamoto-2}.
Infinite gluing techniques (in other words, shadowing lemmas) for other equations can be found in 
Angenent \cite{Angenent}, Eremenko \cite{Eremenko}, Macr\`{i}--Nolasco--Ricciardi \cite{Macri-Nolasco-Ricciardi}
and Gournay \cite{Gournay-thesis, Gournay-interpolation}.

Throughout this section, we fix a positive number $\tau$ such that 
\[ \norm{F_A}_{\mathrm{op}} < d-\tau, \quad d-\tau >1.\]
Let $\delta = \delta(\varepsilon, \tau)>0$ be a sufficiently small number, and 
$T= T(\varepsilon, \tau, \delta)>0$ be a sufficiently large number.
We choose $\delta$ and $T$ so that the following argument works well.

The variable $t$ means the natural projection $t:X\to \mathbb{R}$.

\subsection{Cut and paste} \label{subsection: cut and paste}

Let $I$ be an ASD connection on $E$ defined in Example \ref{example: BPST instanton}.
For $s\in \mathbb{R}$ let $I_s := (-s)^*(I)$ be the pull-back of $I$ by $(-s):E\to E$.
$I_s$ is an ASD connection on $E$ with 
\[  |F(I_s)|_{\mathrm{op}} = \frac{4}{(e^{t-s}+e^{-t+s})^2},\quad  \norm{F(I_s)}_{\mathrm{op}}=1. \]
Most of its energy is contained in a neighborhood of $t=s$.

We define $J\subset \mathbb{Z}$ as the set of $n\in \mathbb{Z}$ satisfying 
$\norm{F_A}_{L^\infty(nT<t<(n+1)T)} < \delta$.
Since $A$ is degenerate, $J$ is an infinite set.
In this subsection we describe a ``cut and paste'' procedure:
We cut and paste the instanton $I_{nT+\frac{T}{2}}$ to $A$ over $[nT,(n+1)T]\times S^3$ for each $n\in J$.
The resulting new connection will be denoted by $B$. 
($B$ is not ASD in general.)

For simplicity of the notation, we suppose $0\in J$, and we explain the cut and paste procedure over 
the region $[0,T]\times S^3$.
Let $\varphi:X\to [0,1]$ be a cut-off function such that 
\[ \varphi = 0 \> \text{on $\{t\leq T/3\}\cup \{t\geq 2T/3\}$}, \quad 
   \varphi = 1 \> \text{on $\{T/3 +1 \leq t\leq 2T/3 -1\}$}. \]
Set $U:=\{T/3-1<t<T/3+2\}\cup \{2T/3-2<t<2T/3+1\}\subset X$.
Since $T\gg 1$ and $\norm{F_A}_{L^\infty(0<t<T)} < \delta \ll 1$,
we can choose connection matrices of $A$ and $I_{T/2}$ over $U$ such that 
\[ \norm{A}_{C^k(U)} < C(k)\delta,\quad \norm{I_{T/2}}_{C^k(U)} < C(k) \delta \quad (\forall k\geq 0).\]
Then we define a connection $B$ on $[0,T]\times S^3$ by 
\[ B := \begin{cases}
        A \quad &\text{on }\{0\leq t\leq T/3\} \cup \{2T/3\leq t\leq T\}\\
        (1-\varphi)A + \varphi I_{T/2} \quad &\text{on } U \\
        I_{T/2} \quad &\text{on } \{T/3+1\leq t\leq 2T/3-1\}.
        \end{cases}  \]     
In the same way we construct a connection $B$ by cutting and pasting the instanton $I_{nT+\frac{T}{2}}$ to $A$ 
over $[nT,(n+1)T]\times S^3$ for every $n\in J$.

Since $\delta\ll 1$ and $\norm{F_I}_{\mathrm{op}} = 1 < d-\tau$, the connection $B$ satisfies 
\begin{equation} \label{eq: op norm of F_B}
 \norm{F_B}_{\mathrm{op}} < d-\tau.
\end{equation}
For $n\not\in J$, we have $B=A$ over $nT\leq t\leq (n+1)T$.
For $n\in J$, we have 
\[ \frac{1}{8\pi^2 T} \int_{nT<t<(n+1)T} |F_A|^2d\vol < \frac{\delta^2 \vol(S^3)}{8\pi^2} < \frac{\varepsilon}{2},\quad 
   (\delta \ll \varepsilon).\]
From this estimate we get 
\begin{equation} \label{eq: energy density of B}
  \rho(B) > \rho(A) -\frac{\varepsilon}{2}.
\end{equation}
Moreover $B$ satisfies the following non-degeneracy condition 
(cf. Lemma \ref{lemma: equivalent condition of non-degeneracy}):
\begin{equation} \label{eq: nondegeneracy of B}
 \begin{split}
  &\norm{F_B}_{L^\infty(nT<t<(n+1)T)} \geq \delta \quad \text{for $n\not \in J$},\\
  &\norm{F_B}_{L^\infty(nT<t<(n+1)T)} \geq \norm{F_I}_{L^\infty(-1<t<1)} \geq 1
  \quad \text{for $n\in J$}.
 \end{split}
\end{equation}
Therefore $B$ satisfies almost all the desired conditions.
The only one problem is that $B$ is not ASD.
But $B$ is an approximately ASD connection:
$F_B^+$ is supported in 
\[ \bigcup_{n\in J} \left(\left\{nT+\frac{T}{3} \leq t \leq nT+\frac{T}{3}+1\right\} \cup 
  \left\{nT+\frac{2T}{3}-1 \leq t\leq nT+\frac{2T}{3}\right\}\right).\]
Since $\delta\ll 1$, 
\begin{equation}  \label{eq: norm of F_B^+}
   \norm{F_B^+}_{L^\infty(X)} \lesssim \delta, \quad \norm{\nabla_B F_B^+}_{L^\infty(X)} \lesssim \delta.
\end{equation}

\subsection{Perturbation} \label{subsection: perturbation}
In this subsection we construct an  ASD connection $\tilde{A}$ by slightly perturbing the connection $B$
constructed in the previous subsection.
We want to solve the equation 
$F^+(B+d_B^*\phi)=0$ for $\phi\in \Omega^+(\ad E)$.
By using the Weitzenb\"{o}ck formula \cite[Chapter 6]{Freed-Uhlenbeck},
\begin{equation*}
 \begin{split}
 F^+(B+d_B^*\phi) &= F_B^+ + d_B^+d_B^*\phi + (d_B^*\phi\wedge d_B^*\phi)^+ \\
    &= F_B^+ + \frac{1}{2}\left(\nabla^*_B\nabla_B + \frac{S}{3}\right)\phi + F_B^+ \cdot \phi 
     + (d_B^*\phi\wedge d_B^*\phi)^+
 \end{split}
\end{equation*}
where $S$ is the scalar curvature of $X=\mathbb{R}\times S^3$.
$S$ is a positive constant.
The following fact on the operator $(\nabla_B^*\nabla_B+S/3)$ is proved in 
\cite[Appendix, Proposition A.7, Lemmas A.1, A.2]{Matsuo-Tsukamoto}.
\begin{lemma} \label{lemma: Green kernel estimate}
For any smooth $\xi\in \Omega^+(\ad E)$ with $\norm{\xi}_{L^\infty} < \infty$, there uniquely 
exists a smooth $\phi\in \Omega^+(\ad E)$ satisfying
\[ \norm{\phi}_{L^\infty} < \infty, \quad \left(\nabla_B^*\nabla_B+\frac{S}{3}\right)\phi = \xi.\]
We will denote this $\phi$ by $(\nabla_B^*\nabla_B+S/3)^{-1}\xi$. 
It satisfies 
\[ |\phi(x)| \leq \int_X g(x,y)|\xi(y)| d\vol(y), \quad \norm{\phi}_{L^\infty} \lesssim \norm{\xi}_{L^\infty}.\]
Here $g(x,y) >0$ is the Green kernel of the operator $\nabla^*\nabla + S/3$
(this is the operator acting on functions). 
It is positive and uniformly integrable:
\[ \int_X g(x,y)d\vol(y) \lesssim 1 \quad (\text{independent of $x$}).\]
Moreover it decays exponentially: For $d(x,y)>1$
\[ 0< g(x,y) \lesssim e^{-\sqrt{S/3}\,d(x,y)} \quad (\text{$d(x,y)$: distance between $x$ and $y$}).\]
\end{lemma}
\begin{lemma} \label{lemma: L^infty estimate of nabla_B phi}
Suppose $\xi\in \Omega^+(\ad E)$ is smooth and $\norm{\xi}_{L^\infty} < \infty$.
Then $\phi := (\nabla_B^*\nabla_B+S/3)^{-1}\xi$ satisfies 
\[ \norm{\phi}_{L^\infty} + \norm{\nabla_B \phi}_{L^\infty} \lesssim \norm{\xi}_{L^\infty}.\]
\end{lemma}
\begin{proof}
$\norm{\phi}_{L^\infty}\lesssim \norm{\xi}_{L^\infty}$ was already given in Lemma \ref{lemma: Green kernel estimate}.
So we want to prove $\norm{\nabla_B \phi}_{L^\infty} \lesssim \norm{\xi}_{L^\infty}$.
From the compactness of $\moduli_d$ (or the Uhlenbeck compactness) and the construction of $B$, 
for any $s \in \mathbb{R}$ 
we can choose a connection matrix of $B$ over $(s,s+1)\times S^3$ satisfying
\[ \norm{B}_{C^k(s<t<s+1)} \lesssim C(k) \quad (\forall k\geq 0).\]
Then from the $L^p$-estimate (Gilbarg--Trudinger \cite[Theorem 9.11]{Gilbarg-Trudinger})
and $\norm{\phi}_{L^\infty}\lesssim \norm{\xi}_{L^\infty}$,
for $1<p<\infty$
\begin{equation} \label{eq: L^p estimate in gluing}
  \norm{\phi}_{L^p_{2,B}(s+1/4<t<s+3/4)} \lesssim C(p)\norm{\xi}_{L^\infty(X)}.
\end{equation}
Then the desired estimate $\norm{\nabla_B \phi}_{L^\infty}\lesssim \norm{\xi}_{L^\infty}$
follows from the Sobolev embedding $L^p_1\hookrightarrow C^0$ ($p>4$).
\end{proof}
Set $\phi := 2(\nabla_B^*\nabla_B+S/3)^{-1}\xi$ where $\xi\in \Omega^+(\ad E)$ is smooth and 
$\norm{\xi}_{L^\infty} < \infty$.
We want to solve the equation 
$F^+(B+d_B^*\phi)=0$, i.e.
\[ \xi = -F_B^+ -F_B^+\cdot \phi -(d_B^*\phi\wedge d_B^*\phi)^+.\]
Set $Q(\xi) := -F_B^+ -F_B^+\cdot \phi -(d_B^*\phi\wedge d_B^*\phi)^+$.
From $\norm{F_B^+}_{L^\infty} \lesssim \delta$ and Lemma \ref{lemma: L^infty estimate of nabla_B phi},
\[ \norm{Q(\xi)-Q(\eta)}_{L^\infty} \lesssim 
  (\delta+\norm{\xi}_{L^\infty}+\norm{\eta}_{L^\infty})\norm{\xi-\eta}_{L^\infty}.\]
Then we can easily check that (when $\delta\ll 1$) the sequence $\{\xi_n\} \subset \Omega^+(\ad E)$ defined by 
\[ \xi_0:=0, \quad \xi_{n+1}:=Q(\xi_n) \]
satisfies $\norm{\xi_n}_{L^\infty}\lesssim \delta$ (the implicit constant is independent of $n$) and
becomes a Cauchy sequence in $L^\infty(X)$.
Let $\xi_n\to \xi_\infty$ in $L^\infty(X)$.
We have $\norm{\xi_\infty}_{L^\infty} \lesssim \delta$.
We will show that $\xi_\infty$ is smooth and satisfies $Q(\xi_\infty)=\xi_\infty$.

Set $\phi_n:=2(\nabla_B^*\nabla_B+S/3)^{-1}\xi_n$. Then 
\begin{equation} \label{eq: recursive relation of xi_n}
 \xi_{n+1} = Q(\xi_n) = -F_B^+-F_B^+\cdot \phi_n -(d_B^*\phi_n\wedge d_B^*\phi_n)^+.
\end{equation}
From the above (\ref{eq: L^p estimate in gluing}) and $\norm{\xi_n}_{L^\infty}\lesssim \delta$,
the sequence $\{\phi_n\}$ is bounded in $L^p_{2,B}(K)$ for every $1<p<\infty$ and compact subset $K\subset X$.
Then from the equation (\ref{eq: recursive relation of xi_n}) the sequence $\{\xi_n\}$ is bounded in $L^p_{1,B}(K)$.
In the same way (the standard bootstrapping argument) we can show that the sequence $\{\xi_n\}$ is bounded in 
$L^p_{k,B}(K)$ for every $k\geq 0$, $1<p<\infty$ and compact subset $K\subset X$.
Therefore $\xi_\infty$ is smooth, and $\xi_n$ converges to $\xi_\infty$ in $C^\infty$ over every compact subset.
Then 
\begin{equation} \label{eq: xi_infty}
    \xi_\infty = -F_B^+-F_B^+\cdot\phi_\infty -(d_B^*\phi_\infty\wedge d_B^*\phi_\infty)^+, \quad
   (\phi_\infty := 2(\nabla_B^*\nabla_B+S/3)^{-1}\xi_\infty).
\end{equation}
Set $\tilde{A}:=B+d_B^*\phi_\infty$.
The connection $\tilde{A}$ is ASD.
The rest of the work is to show that $\tilde{A}$ is non-degenerate and satisfies the condition 
(\ref{eq: condition for tilde A}).

From Lemma \ref{lemma: L^infty estimate of nabla_B phi}, 
$\norm{\phi_\infty}_{L^\infty}+\norm{\nabla_B\phi_\infty}_{L^\infty}\lesssim \norm{\xi_\infty}_{L^\infty}
\lesssim \delta$.
Moreover the equation 
\[ d_B^+d_B^*\phi_\infty + (d_B^*\phi_\infty\wedge d_B^*\phi_\infty)^+ = -F_B^+\]
and $\norm{F_B^+}_{L^\infty}+\norm{\nabla_BF_B^+}_{L^\infty}\lesssim \delta$ (see (\ref{eq: norm of F_B^+}))
implies $\norm{\nabla_B\nabla_B\phi_\infty}_{L^\infty} \lesssim \delta$.
(See the proof of Lemma \ref{lemma: L^infty estimate of nabla_B phi}.)
Hence the curvature 
\[ F(\tilde{A}) = F_B + d_Bd_B^*\phi_\infty + d_B^*\phi_\infty\wedge d_B^*\phi_\infty\]
satisfies $\norm{F(\tilde{A})-F_B}_{L^\infty} \lesssim \delta$.
Since $B$ satisfies $\norm{F_B}_{\mathrm{op}} < d-\tau$ and $\rho(B)>\rho(A)-\varepsilon/2$ 
(see (\ref{eq: op norm of F_B}) and (\ref{eq: energy density of B})), if $\delta =\delta(\varepsilon,\tau)\ll 1$,
we get 
\[ \norm{F(\tilde{A})}_{\mathrm{op}} < d, \quad \rho(\tilde{A}) > \rho(A)-\varepsilon.\]
Therefore $\tilde{A}$ satisfies the condition (\ref{eq: condition for tilde A}).

Finally we show that $\tilde{A}$ is non-degenerate.
It is enough to prove that for all $n\in \mathbb{Z}$ the connection $\tilde{A}$ satisfies 
(see Lemma \ref{lemma: equivalent condition of non-degeneracy})
\begin{equation}\label{eq: non-degeneracy of tilde A}
   \norm{F(\tilde{A})}_{L^\infty(nT<t<(n+1)T)} > \delta/2.
\end{equation}
When $n\in J$, we have $\norm{F_B}_{L^\infty(nT<t<(n+1)T)} \geq 1$
(see (\ref{eq: nondegeneracy of B}))
and $\norm{F(\tilde{A})-F_B}_{L^\infty} \lesssim \delta\ll 1$.
So the above (\ref{eq: non-degeneracy of tilde A}) holds for $n\in J$.

Choose $n\not \in J$. For simplicity, we suppose $n=0$.
From the Green kernel estimate in Lemma \ref{lemma: Green kernel estimate},
\[ |\phi_\infty(x)|\leq 2\int_X g(x,y)|\xi_\infty(y)|d\vol(y) .\]
From (\ref{eq: xi_infty}) and $|F_B^+|,|\phi_\infty|,|\nabla_B\phi_\infty|\lesssim \delta$, 
\[ |\xi_\infty| \lesssim |F_B^+| + \delta^2.\]
Since $0\not\in J$, the distance between $(-1,T+1)\times S^3$ and $\supp(F_B^+)$ is $\gtrsim T$.
The Green kernel $g(x,y)$ decays exponentially. So if we choose $T=T(\varepsilon, \tau,\delta)$ sufficiently large,
then
\[ \norm{\phi_\infty}_{L^\infty(-1<t<T+1)} \lesssim \delta^2.\]
$\phi_\infty$ satisfies the following equation over $(-1,T+1)\times S^3$:
\[ d_B^+d_B^*\phi_\infty = -(d_B^*\phi_\infty\wedge d_B^*\phi_\infty)^+.\]
Since $\norm{d_B^*\phi_\infty\wedge d_B^*\phi_\infty}_{L^\infty}\lesssim 
\norm{\nabla_B\phi_\infty}^2_{L^\infty} \lesssim \delta^2$,
the bootstrapping argument shows 
\[ \norm{\nabla_B\nabla_B\phi_\infty}_{L^\infty(0<t<T)} \lesssim \delta^2.\]
Therefore $|F(\tilde{A})-F_B|\lesssim \delta^2$ over $(0,T)\times S^3$.
Since $\norm{F_B}_{L^\infty(0<t<T)}\geq \delta$ (see (\ref{eq: nondegeneracy of B})) and $\delta\ll 1$, we get 
(\ref{eq: non-degeneracy of tilde A}) for $n=0$.
We have finished the proof of Theorem \ref{thm: abundance of non-degenerate ASD connections}.

\begin{remark} \label{remark: final remark}
If we start with the trivial flat connection $A$ in this gluing argument,
then we can make the argument invariant under the action of the subgroup 
$T\mathbb{Z}\subset \mathbb{R}$. 
Then the resulting non-degenerate ASD connection $\tilde{A}$ becomes periodic
(Example \ref{example: periodic ASD connection}).
So we can conclude that the space $\moduli_d$ $(d>1)$ always contains a non-flat periodic ASD connection.
\end{remark}

\appendix

\section{Another ASD moduli space}
Here we briefly discuss another possibility of the definition of the ASD moduli space.
Let $X=\mathbb{R}\times S^3$ and $E=X\times SU(2)$ as in the main body of the paper.
For $d\geq 0$ we define $\mathcal{N}_d$ as the space of the gauge equivalence classes of ASD connections $A$ on 
$E$ satisfying 
\[ \norm{F_A}_{L^\infty(X)}\leq d.\]
Note that here we use the $L^\infty$-norm, which is different from the operator norm used in the definition of $\moduli_d$.
The space $\mathcal{N}_d$ is endowed with the topology of $C^\infty$-convergence over compact subsets.
$\mathcal{N}_d$ is compact and metrizable, and it admits a natural $\mathbb{R}$-action.
The paper \cite{Matsuo-Tsukamoto} studies the mean dimension and local mean dimension of this $\mathcal{N}_d$.
In particular \cite[Theorem 1.2]{Matsuo-Tsukamoto} shows the following upper bound on the local mean dimension:
\begin{theorem} \label{thm: upper bound on the local mean dimension of N_d}
For any $[A]\in \mathcal{N}_d$,
\[ \dim_{[A]}(\mathcal{N}_d:\mathbb{R}) \leq 8\rho(A).\]
\end{theorem}
If $A$ is an ASD connection on $E$, then the operator norm $|F_A|_{\mathrm{op}}$ and the Euclidean norm 
$|F_A|$ bound each other by 
\[ \frac{1}{\sqrt{3}} |F_A| \leq |F_A|_{\mathrm{op}} \leq |F_A|.\]
(This uses the ASD condition.) Hence 
\[ \mathcal{N}_d \subset \moduli_d \subset \mathcal{N}_{\sqrt{3}\,d}.\]
Then for any $[A]\in \moduli_d$
\[ \dim_{[A]}(\moduli_d:\mathbb{R}) \leq \dim_{[A]}(\mathcal{N}_{\sqrt{3}\,d}:\mathbb{R}) \leq 8\rho(A).\]
This is Theorem \ref{thm: upper bound} in Section \ref{subsection: non-degenerate ASD connections}.
From the knowledge on $\moduli_d$ we can prove the results on $\mathcal{N}_d$:
\begin{theorem}
Let $A$ be a non-degenerate ASD connection on $E$ with $\norm{F_A}_{L^\infty} <d$.
Then 
\[ \dim_{[A]}(\mathcal{N}_d:\mathbb{R}) = 8\rho(A).\]
\end{theorem}
\begin{proof}
We assume that $\moduli_d$ is endowed with a distance and that $\mathcal{N}_d$ is endowed with its restriction.
Then we have 
$B_r([A];\mathcal{N}_d)_{\mathbb{R}} = B_r([A];\moduli_d)_{\mathbb{R}}$
for sufficiently small $r>0$.
Hence by Theorem \ref{thm: local mean dimension around non-degenerate ASD connection}
\[ \dim_{[A]}(\mathcal{N}_d:\mathbb{R}) = \dim_{[A]}(\moduli_d:\mathbb{R}) = 8\rho(A).\]
\end{proof}
\begin{theorem}
Suppose $d>\sqrt{3}$, and let $A$ be an ASD connection on $E$ with $\norm{F_A}_{L^\infty}<d$.
For any $\varepsilon>0$ there exists a non-degenerate ASD connection $\tilde{A}$ on $E$ satisfying 
\[ \norm{F(\tilde{A})}_{L^\infty} < d, \quad \rho(\tilde{A}) > \rho(A)-\varepsilon.\]
\end{theorem}
\begin{proof}
The point is that the instanton $I$ defined in Example \ref{example: BPST instanton} satisfies 
\[ |F_I(t,\theta)| = \frac{4\sqrt{3}}{(e^t+e^{-t})^2}, \quad \norm{F_I}_{L^\infty} =\sqrt{3}.\]
Then the gluing construction in Section \ref{section: gluing infinitely many instantons}
gives the result.
\end{proof}

Let $\rho_{\mathcal{N}}(d)$ be the supremum of $\rho(A)$ over $[A]\in \mathcal{N}_d$.
Let $\mathcal{D}_{\mathcal{N}}\subset [0,+\infty)$ be the set of left-discontinuous points of 
$\rho_{\mathcal{N}}(d)$.
This is at most countable.
From the above theorems, we can prove the following theorem.
(The proof is the same as the proof of Theorem \ref{thm: main theorem}.)
\begin{theorem}
For any $d\in (\sqrt{3},+\infty)\setminus \mathcal{D}_{\mathcal{N}}$,
\[ \dim_{loc}(\mathcal{N}_d:\mathbb{R}) = 8\rho_{\mathcal{N}}(d).\]
\end{theorem}
So if $d>\sqrt{3}$ we have a good understanding of the local mean dimension of $\mathcal{N}_d$.
For $d<1$,
$\mathcal{N}_d=\moduli_d = \{[\text{flat connection}]\}$ is the one-point space
(Example \ref{example: BPST instanton}).
The remaining problem is the case of $1\leq d\leq \sqrt{3}$.
We don't have any good information of this range.

The main good property of the operator norm $\norm{F_A}_{\mathrm{op}}$
is our knowledge of the sharp threshold value described in Example \ref{example: BPST instanton}.

\vspace{10mm}

\address{ Shinichiroh Matsuo \endgraf
Department of Mathematics, Osaka University,
 Toyonaka, Osaka 560-0043, Japan}

\textit{E-mail address}: \texttt{matsuo@math.sci.osaka-u.ac.jp}

\vspace{0.5cm}

\address{ Masaki Tsukamoto \endgraf
Department of Mathematics, Kyoto University, Kyoto 606-8502, Japan}

\textit{E-mail address}: \texttt{tukamoto@math.kyoto-u.ac.jp}

\end{document}